\documentclass[a4paper,10pt]{article}
\usepackage{algorithm2e}

\SetCommentSty{mycommfont} 
\usepackage{aliascnt}
\usepackage[cmex10]{amsmath}
\usepackage{amsfonts}
\usepackage{amssymb}
\usepackage{amsthm}
\usepackage{bbm}
\usepackage{cancel} 
\usepackage{color}
\definecolor{lightred}{RGB}{255,152,152}
\usepackage{dsfont} 
\usepackage{enumerate}
\usepackage{epsfig}
\usepackage{epstopdf}
\usepackage[T1]{fontenc} 
\usepackage{geometry}
\usepackage{graphicx}
\usepackage{hyperref}
\usepackage[latin1]{inputenc}
\usepackage{ifthen}
\usepackage{lscape}
\usepackage{multirow}
\usepackage{multicol}
\usepackage{subfig}
\usepackage{stmaryrd}
\usepackage{tikz}
\usepackage[normalem]{ulem}
\usepackage{verbatim}
\usepackage{version}
\usepackage{xargs}
\usepackage{xr}
\usepackage{authblk}

\newcommand{\NN}{\mathbb{N}}
\newcommand{\PP}{\mathbb{P}}

\newcommand{\RR}{\mathbb{R}}

\newcommand{\ZZ}{\mathbb{Z}}

\def\a{\alpha}

\def\t{\theta}
\def\T{\Theta}

\newcommand{\Da}{ {\cal D }}

\newcommand{\Ffr}{\mathfrak{F}}

\newcommandx{\asvar}[1][1=]{
\ifthenelse{\equal{#1}{}}
{\sigma^2}
{\sigma_{#1}^2}
}
\newcommandx{\asvarti}[1][1=]{
\ifthenelse{\equal{#1}{}}
{\tilde{\sigma}^2}
{\tilde{\sigma}_{#1}^2}
}
\newcommandx{\asquad}[1][1=]{
\ifthenelse{\equal{#1}{}}
{\gamma}
{\gamma_{#1}}
}
\newcommandx{\asquadti}[1][1=]{
\ifthenelse{\equal{#1}{}}
{\tilde{\gamma}}
{\tilde{\gamma}_{#1}}
}

\newcommand{\bmf}{\mathcal{B}_b} 
\newcommand{\bmfone}{\mathcal{B}_1} 
\newcommand{\bg}[2][]{
\ifthenelse{\equal{#1}{}}
{\Psi_{#2}}
{\Psi_{#1,#2}} 
}
\newcommand{\bgx}[2][]{
\ifthenelse{\equal{#1}{}}
{\Psi^X_{\t_{#2},#2}}
{\Psi^X_{\t_{#2}^{#1},#2}} 
}
\newcommand{\bgb}[1]{\overline{\Psi}_{#1}} 

\newcommandx{\cexp}[3][1=]{
\ifthenelse{\equal{#1}{}}
{\mathbb{E}\left[ #2 \mid #3 \right]} 
{\mathbb{E}_{#1}[ #2 \mid #3 ]}
}

\newcommandx{\der}[1][1=]{
\ifthenelse{\equal{#1}{}}
{w}
{w_{#1}}
}

\newcommandx\epart[3][1=]{
\ifthenelse{\equal{#1}{}}
{\xi^\N(#2,#3)}
{\xi_{#1}^\N(#2,#3)}
}
\newcommandx\eparttd[3][1=]{
\ifthenelse{\equal{#1}{}}
{\tilde{\xi}^\N(#2,#3)}
{\tilde{\xi}_{#1}^\N(#2,#3)}
}

\newcommand{\esp}[2][]{
\ifthenelse{\equal{#1}{}}
{\mathbb{E}\left[#2 \right]}
{\mathbb{E}_{#1}\left[#2 \right]}  
} 

\newcommand{\eqsp}{} 
\newcommand{\eqdef}{\triangleq} 

\newcommand{\fb}{\overline{f}} 

\newcommandx\filt[1]{\sigma(\Xb[0],\ldots,\Xb[#1])} 
\newcommandx\filttilde[1]{\sigma(\Xti[0],\ldots,\Xti[#1])}

\newcommandx\filttd[1][1=]{
\ifthenelse{\equal{#1}{}}
{\tilde{\field{F}}^\N}
{\tilde{\field{F}}_{#1}^\N}
}
\newcommandx\ffilt[1][1=]{
\ifthenelse{\equal{#1}{}}
{\field{G}^\N}
{\field{G}_{#1}^\N}
}

\newcommand{\intvect}[2]{\llbracket #1, #2 \rrbracket} 

\newcommand{\kacb}[1]{\overline{\eta}_{#1}} 
\newcommand{\kacti}[1]{\tilde{\eta}_{#1}} 
\newcommand{\kact}[1]{\eta_{#1}^{\Theta}} 
\newcommand{\kacxe}[1]{
\ifthenelse{\equal{#1}{0}}
{\eta^X_{\t_0^{[\Nis]},0}}
{\eta^X_{\t^{[\Nis]}_{0:#1},#1}} 
}
\newcommand{\kacx}[2][]{
\ifthenelse{\equal{#1}{}}
{
\ifthenelse{\equal{#2}{0}}
{\eta^{X}_{\t_0}}
{\eta^{X}_{\t_{0:#2},#2}}
}
{
\ifthenelse{\equal{#2}{0}}
{\eta^{X}_{\t_0^{#1},0}}
{\eta^{X}_{\t^{#1}_{0:#2},#2}}
}
}
\newcommand{\kacX}[1]{
\ifthenelse{\equal{#1}{0}}
{\eta^X_{\T_0,0}}
{\eta^X_{\T_{0:#1},#1}}
}
\newcommand{\kacXN}[2][]{
\ifthenelse{\equal{#1}{}}
{
\ifthenelse{\equal{#2}{0}}
{\eta^{X,\Nin}_{\t_0}}
{\eta^{X,\Nin}_{\t_{0:#2},#2}}
}
{
\ifthenelse{\equal{#2}{0}}
{\eta^{X,\Nin}_{\t_0^{#1},0}}
{\eta^{X,\Nin}_{\t^{#1}_{0:#2},#2}}
}
}

\newcommand{\kacxN}[2][]{
\ifthenelse{\equal{#1}{}}
{
\ifthenelse{\equal{#2}{0}}
{\eta^{X,\Nin}_{\t_0}}
{\eta^{X,\Nin}_{\t_{0:#2},#2}}
}
{
\ifthenelse{\equal{#2}{0}}
{\eta^{X,\Nin}_{\t_0^{#1},0}}
{\eta^{X,\Nin}_{\t^{#1}_{0:#2},#2}}
}
}
\newcommand{\kacXTN}[2][]{
\ifthenelse{\equal{#1}{}}
{
\ifthenelse{\equal{#2}{0}}
{\eta^{X,\Nin}_{\T_0}}
{\eta^{X,\Nin}_{\T_{0:#2},#2}}
}
{
\ifthenelse{\equal{#2}{0}}
{\eta^{X,\Nin}_{\T_0^{#1},0}}
{\eta^{X,\Nin}_{\T^{#1}_{0:#2},#2}}
}
}
\newcommand{\hatkacxN}[2][]{
\ifthenelse{\equal{#1}{}}
{
\ifthenelse{\equal{#2}{0}}
{\eta^{X,\Nin}_{\hat{\t}_0}}
{\eta^{X,\Nin}_{\hat{\t}_{0:#2},#2}}
}
{
\ifthenelse{\equal{#2}{0}}
{\eta^{X,\Nin}_{\hat{\t}_0^{#1},0}}
{\eta^{X,\Nin}_{\hat{\t}^{#1}_{0:#2},#2}}
}
}
\newcommand{\unkacb}[1]{\overline{\gamma}_{#1}}
\newcommand{\unkacti}[1]{\tilde{\gamma}_{#1}}

\newcommand{\unkacx}[1]{
\ifthenelse{\equal{#1}{0}}
{\gamma^X_{\t_0,0}}
{\gamma^X_{\t_{0:#1},#1}}
}

\newcommand{\Lp}{\mathbb{L}^p}
\newcommand{\lip}{labeled island particle}
\newcommandx{\ulim}[1][1=]{
\ifthenelse{\equal{#1}{}}
{\stackrel{}{\longrightarrow}}
{\xrightarrow[{#1} \rightarrow +\infty]{}}
}

\newcommand{\meas}[1]{\set{M}(#1)}
\newcommand{\mf}[1]{\set{F}(#1)}

\newcommand{\N}{N}
\newcommand{\Nis}{N_1}
\newcommand{\Nin}{N_2}

\newcommand{\normdist}{{\sf N}}
\newcommand{\nsetpos}{\mathbb{N}^\ast}  

\newcommand{\1}[1][]{
\ifthenelse{\equal{#1}{}}
{\mathds{1}}
{\mathds{1}_{#1}}
}  

\newcommandx{\osc}[2][1=]{
\ifthenelse{\equal{#1}{}}
{\operatorname{osc}(#2)}
{\operatorname{osc}^2(#2)}
}

\newcommandx{\OscOne}{\operatorname{Osc}_1}

\newcommandx{\op}[1][1=]{
\ifthenelse{\equal{#1}{}}
{Q}
{Q_{#1}}
}

\newcommandx{\plim}[1][1=]{
\ifthenelse{\equal{#1}{}}
{\stackrel{\prob}{\longrightarrow}}
{\xrightarrow[{#1} \rightarrow +\infty]{\prob}}
}

\newcommandx{\dlim}[1][1=]{
\ifthenelse{\equal{#1}{}}
{\stackrel{\Da}{\longrightarrow}}
{\xrightarrow[{#1} \rightarrow +\infty]{\Da}}
}
\newcommand{\psp}{\mathcal{P}} 
\newcommand{\ppart}[2][]{X_{#1}\ifthenelse{\equal{#1}{}}{}{^{#2}}}
\newcommand{\prob}{\mathbb{P}} 
\newcommand{\probmeas}[1]{\mathcal{P}(#1)}  

\newcommand{\pathprobt}[1]{\mathbb{P}^{\Theta}_{\eta^\Theta_0,#1}}
\newcommand{\pathprobx}[1]{\mathbb{P}^{X}_{\theta_{0:#1},#1}}

\newcommand{\rmd}{\mathrm{d}}  
\newcommand{\rset}{\mathbb{R}} 
\newcommand{\rsetpos}{\mathbb{R}^\ast_+}  
 \newcommand{\rsetnonneg}{\mathbb{R}_+} 
 
\newcommand{\set}[1]{\mathsf{#1}} 

\newcommandx{\supn}[2][1=]{
\ifthenelse{\equal{#1}{}}
{\left \| #2 \right \|_\infty} 
{ \| #2 \|_\infty}
}

\newcommand{\separt}[2][]{
\ifthenelse{\equal{#1}{}}
{\xi^{\N}(#2)}
{\xi_{#1}^{\N}(#2)}
}

\newcommand{\separttd}[2][]{
\ifthenelse{\equal{#1}{}}
{\tilde{\xi}^{\N}(#2)}
{\tilde{\xi}_{#1}^{\N}(#2)}
}
\newcommand{\spc}[1]{\set{E}_{#1}} 
\newcommand{\spb}[1]{\overline{\set{E}}_{#1}} 
\newcommand{\spt}[1]{\set{E}^{\T}_{#1}} 
\newcommand{\spx}[2][]{
\ifthenelse{\equal{#1}{}}
{\set{E}^{X}_{#2}}
{\set{E}^{X, #1}_{#2}}
}       
\newcommand{\stsp}{\set{E}}
\newcommand{\stfd}{\mathcal{E}}

\newcommand{\transt}[1]{M^{\Theta}_{#1}}    
\newcommand{\transx}[2][]{
\ifthenelse{\equal{#1}{}}
{M^{X}_{\theta_{#2},#2}}  
{M^{X}_{\theta_{#2}^{#1},#2}} 
}
\newcommand{\transb}[1]{\overline{M}_{#1}}  
\newcommand{\transti}[1]{ \tilde{M}_{#1}}    

\newcommandx{\wgt}[2][1=]{
\ifthenelse{\equal{#1}{}}
{\omega^{#2}}
{\omega^{#2}_{#1}}
}
\newcommandx{\samp}[2][1=]{
\ifthenelse{\equal{#1}{}}
{\xi^{#2}}
{\xi^{#2}_{#1}}
}

\newcommandx{\xb}[1][1=]{
\ifthenelse{\equal{#1}{}}
{\overline{x}}
{\overline{x}_{#1}}
}
\newcommandx{\Xb}[1][1=]{
\ifthenelse{\equal{#1}{}}
{\overline{X}}
{\overline{X}_{#1}}
}
\newcommandx{\Xti}[1][1=]{
\ifthenelse{\equal{#1}{}}
{\tilde{X}}
{\tilde{X}_{#1}}
}

\newcommand{\Yset}{\set{Y}}
\newcommand{\Yfd}{\mathcal{Y}}

\newcommand{\Zset}{\set{Z}}
\newcommand{\Zfd}{\mathcal{Z}}

\newtheorem{thrm}{Theorem}[section]
\newtheorem{prpstn}{Proposition}[section]
\newtheorem{dfntn}{Definition}[section]

\begin{document}
\title{Introduction to labeled island particle models and their asymptotic properties}
\author[1]{C\'ecile Ichard} 
\affil[1]{M\'et\'eo-France-CNRS, CNRM-GAME UMR 3589, 42 Avenue Coriolis, 31057 Toulouse Cedex 1, France; email : cecile.ichard@meteo.fr}

\author[2,3,4]{Christelle Verg\'e}
\affil[2]{Onera - The French Aerospace Lab, Chemin de la Huni\`ere et des Joncherettes, BP 80100, 91123 Palaiseau Cedex, France; email : christelle.verge@onera.fr}
\affil[3]{CNES - Centre National d'Etudes Spatiales, 18 avenue Edouard Belin, 31 401 Toulouse Cedex 9, France}
\affil[4]{CMAP UMR 7641 \'Ecole Polytechnique CNRS, Route de Saclay, 91128 Palaiseau Cedex France}

%
%


%
%
%
\maketitle
\begin{abstract}
Estimation of stochastic processes evolving in a random environment is of crucial importance for example to predict aircraft trajectories evolving in an unknown atmosphere. 
For fixed parameter, interacting particle systems are a convenient way to approximate such stochastic process. 
But the second level of uncertainty provided by the environment parameters  
leads us to also consider interacting particles on the parameter space. 
This novel algorithm is described in this paper. It allows to approximate both a random environment and a stochastic process evolving in this environment, given noisy observations of the process.
It is a sequential algorithm that generalizes island particle models including a parameter. It is referred by us as {\lip} algorithm. 
We prove the convergence of the {\lip} algorithm and we establish $\Lp$ bound as well as time uniform $\Lp$ bound for the asymptotic error introduced by this double level of approximation.
Finally, we illustrate these results on a filtering problem where one learns a dynamical parameter through noisy observations of a stochastic process influenced by the parameter. 

\end{abstract}

\section*{Introduction}
This paper deals with the estimation of stochastic processes whose evolution is influenced by a random environment.
This question is at stake in different areas. In economy, when one wants to estimate the option price with an unknown volatility \cite{cont:2006} using the Black-Scholes model, one can consider that the option price has its evolution influenced by an unknown environment, the market volatility. In biology, when one wants to estimate the number of bacteria whereas the environment factors are unknown \cite{Augustin:carlier:2000}, one can model the evolution of the bacteria number as a stochastic process whose evolution is influenced by unknown external factors. In air traffic management, this modelization can also be used when one wants to predict aircraft trajectories evolving in an unknown atmosphere.
Indeed if pilot intents and some aircraft parameters are not known, actual wind and temperature evolve locally and are not perfectly known neither. Those atmospheric parameters which appear in the dynamic equations of the aircraft have a great importance to predict the future position of the aircraft. They are thus both uncertain. Therefore, in order to improve the trajectory prediction, one has to learn aircraft parameters but also atmospheric ones. It has been shown in \cite{ichard:baehr:2013} that it can be done using mode-S radar observations and this specific model.

When the stochastic process evolving in the random media is Gaussian and its evolution is linear, the double estimation can be made using interacting Kalman filters (IKF) \cite{delmoral:2004, zghal:mevel:delmoral:2014}.
However when the dynamics are non-linear, as for aircraft dynamics, an analytic resolution is not possible. A method based on interacting particle systems, which takes into account the randomness due to the environment and also the randomness coming from the process itself, was proposed by Del Moral in \cite{delmoral:2004}. This idealized algorithm would be a sequential Monte Carlo (SMC) algorithm on the couple defined by the random environment and the conditional law of the process evolving in this random environment given the history of the environment. However, the calculation of the previous conditional law is not tractable in practice when the dynamics are non linear. Therefore another approximation level is necessary in order to estimate this conditional law. 
We propose in this paper to use interacting systems of interacting particles. These interacting systems can be viewed as a two-level interacting particle system. The top level particles are composed of an environment proposition and an empirical measure which gives an approximation of the process law  evolving in the proposed environment. The empirical measure is obtained by the second level of interacting particles. This nested structure was also presented in \cite{baehr:2008} for mean field processes.

This algorithm can be seen as a generalization of interacting island particle models where each island is associated with a random parameter. 
Those island particle models have been introduced in \cite{verge:dubarry:delmoral:moulines:2013} and their statistical properties studied in \cite{verge:delmoral:moulines:olsson:2014}, but without parameters. The first paper deals with the parallelization of interacting particle systems, the second one studies the asymptotic properties of the ensuing estimator.
Concerning filtering problems, Chopin et al. in \cite{chopin:jacob:papaspiliopoulos:2013} introduced a kind of island particle models where each island is identified by a parameter proposition.
They proposed an algorithm called $\text{SMC}^2$ which is a practical version of the idealized iterated batch importance sampling (IBIS) algorithm introduced by Chopin in \cite{chopin:2002} for exploring a sequence of parameter posterior distributions.
The considered parameter did not have any proper dynamic whereas in the present paper the stochastic process evolution scheme depends on a dynamic parameter.
Moreover, in the $\text{SMC}^2$ algorithm, islands of particles grow continuously with time as particles ancestral lines are required to estimate the likelihood increments, and by their product to estimate the total likelihood. 
The algorithm introduced by Crisan et al. in \cite{crisan:miguez:2013} is a different version of the $\text{SMC}^2$ which allows also the estimation of fixed parameters of a state-space dynamic system using sequential Monte Carlo methods. However, unlike the $\text{SMC}^2$ method, the proposed algorithm by Crisan et al. operates in a purely sequential and recursive manner. In particular, the scheme for the rejuvenation of the particles in the parameter space is simpler, given that it does not need the simulation of the auxiliary particle filter from initial time to evaluate the likelihood.
Therefore the algorithm we propose in this paper is similar to the algorithm of \cite{crisan:miguez:2013} in the sense that it is sequential in time and structured as a nested interacting particle filters, but different as it deals with dynamic parameters.   

In this article, we present a novel algorithm for estimating both a random environment and a process whose evolution depends on this environment, and study the asymptotic properties of the ensuing estimators. This study is of great importance to justify the convergence of this algorithm and also a challenging issue as it deals with error in distribution space. 
Therefore as a first step we establish $\Lp$ bound for the asymptotic error introduced by this double level of approximation at every time step. As a matter of fact, the shape of the bound was predicted by Baehr in his thesis \cite{baehr:2008}. Then we obtain a time uniform $\Lp$ bound for the error. From there we deduce the almost sure convergence of the estimator towards the target measure. 
Afterwards, we compare the {\lip} algorithm and interacting Kalman filters (IKF) on a filtering example dealing with the evolution of a mobile on a random media. In particular, it appears that the {\lip} algorithm gives a better estimate of the position and the speed of the mobile than IKF. 
Finally, the {\lip} algorithm is applied to another filtering problem where one learns a dynamical parameter through observations of a stochastic process influenced by the parameter. The theoretical results of this paper are illustrated on this example.

Formalization of the problem through Feynman-Kac measures is given in \autoref{sec:FK:random:media}, then the {\lip} algorithm is described in \autoref{sec:algo}. $\Lp$ bounds of this algorithm are established in \autoref{sec:lpbounds}. Finally, convergence of the {\lip} algorithm and some results proved in \autoref{sec:asymptotic} are illustrated in 
\autoref{sec:application} on two filtering examples.
\section{Feynman-Kac models in random media}
\label{sec:FK:random:media}

In this section, we first present an example which motivates our study, and then we introduce notations and models.

\subsection{Example of process evolution in random media}
\label{example:model}
In this article, one always consider stochastic processes whose evolution are influenced by their surrounding environment. When the environment is unknown, one can be interested in estimating both the environment and the law of the stochastic process itself using observations of the last one. Take a really simple example : a mobile evolving in $\RR^2$ whose dynamics is influenced by an unknown exterior force. 
This problem can be modeled by the following system of equations
\begin{align}
\label{mobile}
 \left\lbrace\begin{array}{lll}     
              X_{n+1} &=& X_{n} + V_n \left(\begin{array}{c}\cos \a \\
              \sin \a \end{array}\right)\Delta t + \Theta_{n+1} \Delta t + B^X_n\\
	      V_{n+1} & = & V_n + B^V_n \eqsp,\\
             \end{array}
\right.
\end{align}
where $X_{n+1}$ denotes the position of the mobile which depends on $\Theta_{n+1}$ and $B^X_n$ a Gaussian noise. The proper speed $V_n$ of the mobile, is known up to a Gaussian white noise $B^V_n$. $\a$ is the course track parameter of the mobile. The vector $\Theta_{n+1}$ is random and represents the unknown force acting on the position of the mobile.\\
We are interested in the estimation of the state of the mobile, which depends on the parameter $\Theta_{n+1}$. 
We thus need to learn both the force, the speed and the position of the mobile.
Consider now that noisy observations $Y_{n}$ of the mobile's state are available. 
One has to estimate the quantity $\cexp{(X_0,V_0,\Theta_{0}),\ldots,(X_n,V_n,\Theta_{n})}{ Y_0,\ldots,Y_n} \eqsp.$
Therefore, we need to use a model which can tackle this issue. To this end, the formalism of Feynman-Kac models in random media is well adapted. 
In \autoref{subsec:FKmodel}, we recall the definitions attached to this model and some important results. For a more detailed review see \cite{delmoral:2004}.

\subsection{Notations}
Let us define some notations used in this paper. 
For $(m, n) \in \ZZ^2$ such that $m \leq n$ we denote $\intvect{m}{n} \eqdef \{m, m + 1, \ldots, n \} \subset \ZZ$. 
We will use the vector notation $a_{m:n} \eqdef (a_m, \ldots, a_n)$.
Moreover, $\rsetnonneg$ and $\rsetpos$ denote the sets of nonnegative and positive real numbers respectively, and $\nsetpos$ the set of positive integers. 

$\normdist(\mu, \Sigma)$ denotes a multivariate Gaussian distribution with mean $\mu$ and covariance matrix $\Sigma$.  

In the sequel we assume that all random variables are defined on a common probability space $(\Omega,\mathcal{F},\PP)$. For some given measurable space $(\stsp, \stfd)$ we denote by $\meas{\stsp}$ and $\probmeas{\stsp} \subset \meas{\stsp}$ the set of measures and probability measures on $(\stsp, \stfd)$, respectively. In addition, we denote by $\mf{\stsp}$ the set of real-valued measurable functions on $(\stsp, \stfd)$ and by $\bmf(\stsp) \subset \mf{\stsp}$ the set of bounded such functions. For any $\nu \in \meas{\stsp}$ and $f \in \mf{\stsp}$ we denote by $\nu f \eqdef \int f(x) \, \nu(\rmd x)$  the Lebesgue integral of $f$ under $\nu$ whenever this is well-defined. Now, given also some other $(\Yset, \Yfd)$ measurable space, an \emph{unnormalized transition kernel} $K$ from $(\stsp, \stfd)$ to $(\Yset, \Yfd)$ is a mapping from $\stsp \times \Yfd$ to $\rset$ such that for all $\set{A} \in \Yfd$, $x \mapsto K(x, \set{A})$ is a nonnegative measurable function on $\stsp$ and for all $x \in \stsp$, $\set{A} \
mapsto K(x, \set{A})$ is a measure on $(\Yset, \Yfd)$. If $K(x, \Yset) = 1$ for all $x \in \stsp$, then $K$ is called a \emph{transition kernel} (or simply a \emph{kernel}). The kernel $K$ induces two integral operators, one acting on functions and the other on measures. More specifically, let $f \in \mf{\stsp}$ and $\nu \in \meas{\stsp}$ and define the measurable function 
$$
	K f : \stsp \ni x \mapsto \int f(y) \, K(x, \rmd y) \eqsp,
$$
and the measure 
$$
	\nu K : \Yfd \ni \set{A} \mapsto \int K(x, \set{A}) \, \nu(\rmd x) \eqsp,
$$
whenever these quantities are well-defined. Finally, let $K$ be as above and let $L$ be another unnormalized transition kernels from $(\Yset, \Yfd)$ to some third measurable space $(\Zset, \Zfd)$; then we define the \emph{product} of $K$ and $L$ as the unnormalized transition kernel
$$
K L : \stsp \times \Zfd \ni (x, \set{A}) \mapsto \int K(x, \rmd y) \, L(y, \set{A}) \eqsp,
$$
from $(\stsp, \stfd)$ to $(\Zset, \Zfd)$ whenever this is well-defined.
 
\subsection{Introduction of Feynman-Kac models}
\label{subsec:FKmodel}
Let $\Theta_n$ be a random process on $\spt{n}$ which influences the evolution of another random process $X_n$ on $\spx{n}$. 
In order to avoid any confusion, all the quantities which refer to the random process $\Theta_n$ (resp. $X_n$) may be identified by the exponent $\Theta$ (resp. $X$).
Let the couple $(\Theta_n,X_n)_{n \in \NN}$ be a $\spc{n} \eqdef (\spt{n},\spx{n})$- valued Markov chain of elementary transition matrix $T_n$ form $\spc{n-1}$ to $\spc{n}$ defined by 
$$T_n\left((\t_{n-1},x_{n-1}),\rmd (\t_n,x_n)\right) \eqdef \transt{n}(\t_{n-1},\rmd \t_n) \transx{n}(x_{n-1},\rmd x_n) \eqsp,$$
where $\transt{n}$ and $\transx{n}$ are the transition kernels of the $\Theta_n$ and $X_n$ processes from $\spt{n-1}$ to $\spt{n}$ and from $\spx{n-1}$ to $\spx{n}$ respectively.\\
Its initial distribution is given by 
$$\eta_{0}(\rmd(\t_0,x_0))  \eqdef \kact{0}(\rmd \t_0)\kacx{0}(\rmd x_0) \eqsp,$$
with $\kact{0} \in \psp(\spt{0})$ and $\kacx{0} \in \psp(\spx{0})$, denoting respectively the initial distributions of $\T_0$ and $X_0$ given $\T_0=\t_0$.\\
Let $(G_n)_{n \in \NN}$ be a collection of bounded measurable functions from $\spc{n}$ to $]0,\infty[$. We define the Feynman-Kac measure associated to the couple $(G_n,T_n)$ with initial distribution $\eta_0$ by 
\begin{multline}
 \mathbb{Q}_{\eta_0,n} \left(\rmd \left((\t_0,x_0),\ldots,(\t_n,x_n)\right)\right) \\
 \eqdef \frac{1}{\mathcal{Z}_n}\left\{\prod_{p=0}^{n-1} G_p(\t_p,x_p)\right\} \pathprobt{n} \left(\rmd (\t_0,\ldots,\t_n)\right)
 \pathprobx{n} \left(\rmd(x_0,\ldots,x_n)\right) \eqsp,\
\end{multline}
with the normalizing constant $\mathcal{Z}_n$, given by 
$$ \mathcal{Z}_n \eqdef \esp[\eta_0]{ \prod_{p=0}^{n-1} G_p(\Theta_p,X_p) } > 0,$$
and the two path probabilities
$$ \pathprobt{n} \left(\rmd(\t_0,\ldots,\t_n)\right) \eqdef \kact{0}(\rmd \t_0) \transt{1} (\t_0,\rmd\t_1) \ldots \transt{n}(\t_{n-1},\rmd \t_n),$$
and
$$ \pathprobx{n} \left(\rmd(x_0,\ldots,x_n)\right) \eqdef \kacx{0}(\rmd x_0) \transx{1}(x_0,\rmd x_1) \ldots \transx{n}(x_{n-1},\rmd x_n).$$
As one may have noticed, given $\T_{0:n}=\t_{0:n}$, the sequence $X_n$ is also a Markov chain of transition kernels $(\transx{n})_{n\in \nsetpos}$ and initial distribution $\kacx{0}$. Then one can associate to it another Feynman-Kac path measure which is called quenched.
\begin{dfntn}
\label{quenched Feynman-Kac path measure}
The quenched Feynman-Kac path measure associated to the realization $\Theta_{0:n}=\t_{0:n}$ is defined by 
$$\mathbb{Q}^X_{\t_{0:n},n} \left(\rmd(x_0,\ldots,x_n)\right) 
\eqdef \frac{1}{\mathcal{Z}^X_{\t_{0:n},n}} \left\{ \prod_{p=0}^{n-1} G_p(\t_p,x_p)\right\} \pathprobx{n} \left(\rmd(x_0,\ldots,x_n)\right) \eqsp,$$
where the quenched normalizing constant $\mathcal{Z}^X_{\t_{0:n},n}$ is given by 
$$\mathcal{Z}^X_{\t_{0:n},n} \eqdef \esp[\t_{0:n}]{ \prod_{p=0}^{n-1} G_p(\t_p,X_p)}>0.$$
\end{dfntn}
~\\
In the rest of the paper the quenched potential functions are denoted by $G_{\t_p,p}$ and defined as
\begin{equation} \label{eq:qpot}
G_{\t_p,p} : x_p \in \spx{p} \mapsto G_{\t_p,p}(x_p) \eqdef G_p(\t_p,x_p).
\end{equation}
To get further into the dynamic, one can define the time marginal of the quenched Feynman-Kac measure also called the quenched Feynman-Kac distribution.
\begin{dfntn}
\label{quenched Feynman-Kac flow}
For every realization $\Theta_{0:n}=\t_{0:n}$, the quenched Feynman-Kac distribution flow $\kacx{n}$ on $\spx{n}$ is defined for every $f_n \in \bmf(\spx{n})$ by 
$$\kacx{n} (f_n) \eqdef \unkacx{n} (f_n) / \unkacx{n}(\1)$$
with $\unkacx{n}(f_n) \eqdef \esp[\t_{0:n}]{f_n(X_n) \prod_{p=0}^{n-1}G_{\t_p,p}(X_p)} \eqsp.$
\end{dfntn}
The distribution of $X_n$ depends on the trajectory $\t_{0:n}$ which is emphasized by denoting the unnormalized quenched Feynman-Kac distribution by $\unkacx{n}$. 
An important result taken from [\cite{delmoral:2004}, Proposition 2.6.2] is recalled below.
\begin{prpstn}
 The quenched distribution sequence $(\kacx{n})_{n \in \NN}$ satisfies the non linear equation :
\begin{equation} \label{rec}
 \kacx{n+1} = \bgx{n}(\kacx{n}) \transx{n+1} \eqsp, 
\end{equation} 
where the mapping $\bgx{n}:\psp(\spx{n}) \rightarrow \psp(\spx{n+1})$ is given by 
\begin{equation} \label{def:bgx}
\bgx{n}(\kacx{n})(\rmd x_{n}) \eqdef \frac{1}{\kacx{n} (G_{\t_{n},n})} G_{\t_{n},n}(x_{n}) \kacx{n}(\rmd x_{n}) \eqsp. 
\end{equation}
Defining the mapping $\Phi^X_{n+1}$ by 
 \begin{align} \label{quenched non linear equation}
 \begin{array}{r c l}
  \Phi^X_{n+1} : \left(\spt{n} \times \spt{n+1} \right)\times \psp(\spx{n})& \rightarrow & \mathcal{P}(\spx{n+1})\\
  \left( (\t_{n},\t_{n+1}),\kacx{n} \right) & \mapsto & \bgx{n}(\kacx{n}) \, \transx{n+1} 
  \end{array}
\end{align}
The non linear recursion \eqref{rec} can be reformulated as 
\begin{equation} \label{eq:nonlin}
\kacx{n+1} =  \Phi^X_{n+1}\left( (\t_{n},\t_{n+1}),\kacx{n} \right) \eqsp.
\end{equation}
\end{prpstn}
Remind that, for a fixed value $\theta_{0:n}$ of the random process $\Theta_{0:n}$, the probability measures $(\kacx{n})_{n\in \NN}$ can be approximated recursively thanks to an interacting particle system which evolves successively according to selection step with potentials $G_{\t_n,n}$ defined in \eqref{eq:qpot} and transition kernels $\transx{n}$.
See \cite{delmoral:2004} for further details.
Now, consider that the random environment $\Theta_{0:n}$, where the stochastic process $X_n$ evolves, is not known. Then we focus our interest on the estimation of the couple 
\begin{equation} \label{def:Xb}
\overline{X}_n \eqdef (\Theta_n,\kacX{n}) \in \spb{n} \eqdef (\spt{n} \times \psp (\spx{n}) ) \eqsp,
\end{equation}
made up of the environment and the law of the process evolving in this environment. The tricky part will be to deal with the probability measure space.
First, notice that, as it has been shown in \cite{delmoral:2004}, the pair process is a Markov chain.
\begin{prpstn}[\cite{delmoral:2004}, Proposition 2.6.3] \label{prop:MC_distrib}
 $\overline{X}_n$ is a Markov chain with transition kernel $\transb{n}$ defined for every function $\overline{f}_n \in \bmf(\spb{n})$ and $(u,\eta) \in \spb{n}$ by 
 $$\transb{n}(\overline{f}_n)(u,\eta) \eqdef \int_{\spt{n}} \transt{n} (u,\rmd v) \overline{f}_n(v, \Phi^X_n((u,v),\eta) )$$
 and with initial distribution $\overline{\eta}_0 \in \psp (\spb{0})$ defined by 
 $$\kacb{0}(\rmd(u,\nu)) \eqdef \kact{0} (\rmd u) \delta_{\kacx{0}} (\rmd \nu) \eqsp. $$
\end{prpstn}
To this Markov chain, one may associate the Feynman-Kac distribution flow $\kacb{n}$ defined for every $\fb_n \in \bmf (\spb{n})$ by 
\begin{align}\label{Feynman-Kac flow in distribution}
\kacb{n}(\fb_n) \eqdef \unkacb{n} (\fb_n) / \unkacb{n} (\1)
\end{align}
where
$$ \unkacb{n} (\fb_n) \eqdef \esp[\kacb{0}]{\fb_n(\overline{X}_n) \prod_{p=0}^{n-1}\overline{G}_p(\overline{X}_p)} \eqsp,$$
and the functions $\overline{G}_p$ are non negative functions defined as follows :
\begin{align}\label{potential on distribution}
\begin{array}{r c l}
\overline{G}_p : \spb{p} & \rightarrow & [0,\infty[\\
  (u,\eta) & \mapsto & \overline{G}_p(u,\eta) = \int_{\spx{p}} G_p(u,x)\eta(\rmd x) = \int_{\spx{p}} G_{u,p}(x) \eta(\rmd x) = \eta(G_{u,p}) \eqsp.
\end{array}
\end{align}
\begin{prpstn}[\cite{delmoral:2004}, p. 86] \label{prop:non linear equation distribution}
For all $n \in \NN$, the sequence $\kacb{n}$ satisfies the following non linear recursive equation :
\begin{equation} \label{eq:recb}
\kacb{n+1} = \bgb{n}(\kacb{n}) \transb{n+1} = \overline{\Phi}_{n+1} (\kacb{n}) \eqsp,
\end{equation}
where for every $\mu \in \psp (\spb{n})$, the application $\bgb{n} : \psp (\spb{n}) \rightarrow \psp (\spb{n})$, is defined by
\begin{equation} \label{eq:bgb}
  \bgb{n}(\mu)(\fb_n) = \mu(\overline{G}_n \fb_n) / \mu(\overline{G}_n)\eqsp,  \quad (\forall \fb_n \in \bmf (\spb{n})) \eqsp,
\end{equation}
and the operator $\overline{\Phi}_n$ is defined by 
 \begin{align*}
 \begin{array}{r c l}
  \overline{\Phi}_{n+1} : \psp (\spb{n}) & \rightarrow & \psp(\spb{n+1}) \\
  \mu & \mapsto &  \bgb{n}(\mu) \transb{n+1} \eqsp.
  \end{array}
 \end{align*}
\end{prpstn}
In the non linear case, \eqref{eq:recb} cannot be solved analytically. Therefore, in the next section, we introduce an interacting particle system to approximate recursively the sequence of Feynman-Kac probability measures $(\kacb{n})_{n \in \NN}$.


\section{Algorithm derivation} \label{sec:algo}
This section is about the algorithm associated with the Feynman-Kac distribution flow $\kacb{n}$ defined in \eqref{Feynman-Kac flow in distribution}.
One considers the process $\Xb[n]$ associated with the pair $(\overline{G}_n,\transb{n})$, where the transition kernel $\transb{n}$ is defined in \autoref{prop:MC_distrib} and the potential function $\overline{G}_n$ is defined in \eqref{potential on distribution}.

\subsection{Idealized interacting particle model} \label{sec:IKF}
Let $\Nis$ be some positive integer. A $\Nis$-interacting particle system  associated with the sequence $((\overline{G}_n,\transb{n}))_{n \in \NN}$ and the initial distribution $\kacb{0}$, is a sequence of non-homogeneous Markov chain, denoted by $ \Xb[n]^{[\Nis]}$, taking value in the product space $\spb{n}^{\Nis}$,
$$ \Xb[n]^{[\Nis]} \eqdef (\Xb[n]^{i})_{i=1}^{\Nis} = (\Xb[n]^{1},\ldots,\Xb[n]^{\Nis}) \in \spb{n}^{\Nis}\eqdef\underbrace{\spb{n}\times\ldots\times\spb{n}}_{\Nis\text{ times}} \eqsp. $$
The initial state of the Markov chain $\Xb[0]^{[\Nis]}$ consists in $\Nis$ independent random variables with common distribution $\kacb{0}$.
The interacting particle system $(\Xb[n]^{i})_{i=1}^{\Nis} $ explores the state space $\spb{n}$ and with the dynamic given to it, empirically samples the law $\kacb{n}$. 
Each particle $i$ of the system consists in a random variable $\Xb_n^i=(\theta^i_n,\kacx[i]{n}) \in \spb{n} \eqsp.$
Therefore, the empirical process $\kacb{n}^{\Nis}$ is defined by
\begin{equation}\label{def:kacb}
 \kacb{n}^{\Nis}\eqdef \frac{1}{\Nis}\sum_{i=1}^{\Nis}\delta_{\Xb[n]^i}.
\end{equation}
 The elementary transition of the Markov chain $ \Xb[n]^{[\Nis]}$ from $\spb{n}^{\Nis}$ to $\spb{n+1}^{\Nis}$ is given for any $\xb[n]^{[\Nis]} \eqdef (\xb[n]^1, \ldots, \xb[n]^{\Nis}) \in \spb{n}^{\Nis}$ by 
\begin{align*}
\mathbb{P}^{\Nis}_{\kacb{0}} \left(\Xb[n+1]^{[\Nis]} \in \rmd \xb[n+1]^{[\Nis]} \, | \, \Xb[n]^{[\Nis]}\right)
&\eqdef \prod_{i=1}^{\Nis} \overline{\Phi}_{n+1} (\kacb{n}^{\Nis}) (\rmd \xb[n+1]^i) \\ 
&= \prod_{i=1}^{\Nis} \sum_{j=1}^{\Nis} \dfrac{\overline{G}_n(\Xb[n]^j) }{\sum_{k=1}^{\Nis} \overline{G}_n(\Xb[n]^k)} \transb{n+1} (\Xb[n]^j, \rmd \xb[n+1]^i) \eqsp,\text{ thanks to \eqref{eq:recb}} \eqsp.
\end{align*}
Thus, the evolution of the particle swarm consists in two steps : a selection and a mutation. 
In the selection step, the particles $(\Xb[n]^i)_{i=1}^{\Nis}$ are selected multinomially with probability proportional to their potentials $(\overline{G}_n(\Xb[n]^i))_{i=1}^{\Nis}$. Selected particles are identified with a hat on \autoref{particle exact measure}. Then the mutation step is performed independently using the kernel $\transb{n+1}$. The evolution scheme of the particles is illustrated on \autoref{particle exact measure}.
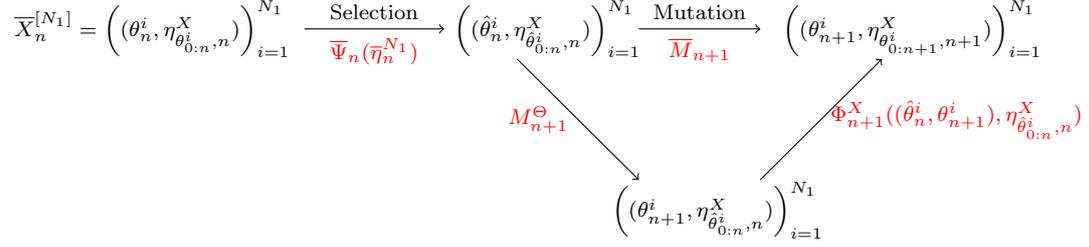
\begin{figure}[!h]
 \begin{tikzpicture}[scale=0.8]
 \draw (0,0) node {\footnotesize{$\Xb[n]^{[\Nis]} = \left((\t^{i}_{n},\kacx[i]{n})\right)_{i=1}^{\Nis}$}};
 \draw (6.5,0) node {\footnotesize{$\left((\hat{\t}^{i}_{n},\eta^{X}_{\hat{\theta}^{i}_{0:n},n})\right)_{i=1}^{\Nis}$}};
 \draw (12.5,0) node {\footnotesize{$\left((\t^{i}_{n+1},\kacx[i]{n+1})\right)_{i=1}^{\Nis}$}};
 \draw [->] (2.5,0) -- (4.8,0) node [midway,above] {\footnotesize{Selection}} node[midway, below] {\footnotesize{\color{red}$\bgb{n}(\kacb{n}^{\Nis})$}};
 \draw [->] (8,0) -- (10,0) node [midway, above] {\footnotesize{Mutation}}  node[midway, below] {\footnotesize{\color{red} $\transb{n+1}$}};
 \draw (9.3,-3) node {\footnotesize{$\left((\t^{i}_{n+1},\eta^{X}_{\hat{\theta}^{i}_{0:n},n})\right)_{i=1}^{\Nis}$}};
 \draw [->] (6,-0.5) -- (8,-2.5) node [midway, left] {\footnotesize{{\color{red}$\transt{n+1}$}}};
 \draw [->] (10,-2.5) -- (12,-0.5) node [midway, right] {\footnotesize{{\color{red}$\Phi_{n+1}^X((\hat{\t}^{i}_{n},\t^{i}_{n+1}),\eta^{X}_{\hat{\theta}^{i}_{0:n},n})$}}};
 \end{tikzpicture}
\caption{Evolution scheme of the interacting particle system for exact measures.}
\label{particle exact measure}
\end{figure}
Using this algorithm one can empirically sample the measure $\kacb{n}$ at each time step $n$. Several results are available to qualify the subsequent estimator. However, as one may have noticed, for each $\t^i_n$ the measure $\kacx[i]{n}$ corresponds to the quenched distributions defined in \eqref{quenched Feynman-Kac flow}. That means that one should have the exact quenched measure associated with the parameter realization $\t^i_{0:n}$ to use that standard particle algorithm. This can happen in two special cases.

Firstly, one special case is when the transition kernel $\transx{n}$ is Gaussian and the initial distribution $\kacx{0}$ is Gaussian. Indeed, it turns out that this particle algorithm corresponds to the interacting Kalman filters (IKF) (see \cite{delmoral:2004}, \cite{zghal:mevel:delmoral:2014}). That is a $N_1$-interacting particle model which is composed of $N_1$ particles where the measure value part are Gaussian distributions. In other words, for each particle $\t^i_n$, one iterative step of the Kalman filter is run to update the measure, \textit{i.e.} one prediction step and one correction step. Those filters are then competing through the selection step using the transformation $\bgb{n}$ defined in \eqref{eq:bgb}.
For example let us consider the case where $\Theta_{n}$ is a $\spt{n}$-process with initial distribution $\kact{0}$ and elementary transition kernel $\transt{n}$. For a realization $\t_{0:n}$ of  $\Theta_{0:n}$, consider that $(X_n,Y_n)$ is a $\mathbb{R}^{p+q}$-Markov chain, for positive integers $(p, q)$, defined through the linear Gaussian system :
\begin{align*}
 \left\lbrace
 \begin{array}{l l l}
  X_n &=& A_{\t_n,n} \, X_{n-1}+a_{\t_n,n} + B_{\t_n,n} \, \varepsilon^X_n\eqsp, \quad n\geq 1 \\
  Y_n &=& C_{\t_n,n} \, X_n+c_{\t_n,n} + D_{\t_n,n} \, \varepsilon^Y_n\eqsp, \quad n\geq 0 \eqsp.
 \end{array}
 \right.
\end{align*}
$(A_{\t_n,n}, B_{\t_n,n}, C_{\t_n,n}, D_{\t_n,n})$ and $(a_{\t_n,n}, c_{\t_n,n})$ are respectively matrices and deterministic vectors of appropriate dimension which may depend on a parameter $\t_n$.
The sequences $\varepsilon^X_n$ and $\varepsilon^Y_n$ are two independent white noises, independent from the initial condition $X_0$. 
There are Gaussian random variables whose mean and variance are given by
$$ X_0 \sim\normdist(m_{\t_0,0},\Sigma_{\t_0,0}) \eqsp, \quad \varepsilon^X_n \sim \normdist(0,\Sigma^X_n)\eqsp, \quad \text{and}~ \varepsilon^Y_n \sim \normdist(0,\Sigma^Y_n) \eqsp.$$
In this framework, $\kacx{n}$ corresponds to the conditional law of $X_n$ given the observations $Y_{0:n-1} =y_{0:n-1}$ and the history of the parameter $\theta_{0:n}$, also called optimal predictor. 
One wants to estimate recursively the law of the couple $(\Theta_n,\kacx{n})$ using observations $Y_{0:n-1} =y_{0:n-1}$.
For that purpose, one needs to introduce the optimal filtering which is the conditional law of $X_n$ given the observations $Y_{0:n} =y_{0:n}$ and the history of the parameter $\theta_{0:n}$.
It turns out that these previous distributions are Gaussian respectively denoted by $\kacx{n}=\normdist(m_{\t_n,n},\Sigma_{\t_n,n})$ and $\normdist(\hat{m}_{\t_n,n},\hat{\Sigma}_{\t_n,n})$.
Thus, 
\begin{align*}
\hat{m}_{\t_{n},n} &= \cexp[\t_{0:n}]{X_{n}}{Y_{0:n}} \\
\hat{\Sigma}_{\t_{n},n} &= \esp[\t_{0:n}]{(X_{n}-\hat{m}_{\t_{n},n})(X_{n}-\hat{m}_{\t_{n},n})^T} \\
m_{\t_{n+1},n+1} &= \cexp[\t_{0:n+1}]{X_{n+1}}{Y_{0:n}} \\
\Sigma_{\t_{n+1},n+1} &= \esp[\t_{0:n+1}]{(X_{n+1}-m_{\t_{n+1},n+1})(X_{n+1}-m_{\t_{n+1},n+1})^T} \eqsp.
\end{align*}
Moreover, the mapping $\Phi^X_{n+1}$ defined in \eqref{quenched non linear equation} which is used to update the measure valued part $\kacx{n}$ corresponds to a complete step of the Kalman filter evolution between predictors. This means that 
$\Phi^X_{n+1}((\t_n,\t_{n+1}),\normdist(m_{\t_n,n},\Sigma_{\t_n,n}))$ is also a Gaussian distribution whose mean and covariance matrix are obtained recursively through two steps:
$$
\normdist(m_{\t_n,n},\Sigma_{\t_n,n}) \xrightarrow[]{\text{Correction}} 
\normdist(\hat{m}_{\t_n,n},\hat{\Sigma}_{\t_n,n}) \xrightarrow[]{\text{Prediction}} 
\normdist(m_{\t_{n+1},n+1},\Sigma_{\t_n,n}) \eqsp.
$$
The first one is a correction step which is given by 
$$\left\lbrace\begin{array}{lll}
\hat{m}_{\t_n,n} &=& m_{\t_n,n} + K_{\t_n,n}(Y_n-(C_{\t_n,n} m_{\t_n,n} + c_{\t_n,n}))\\
\hat{\Sigma}_{\t_n,n} &=& (I - K_{\t_n,n} C_{\t_n,n}) \Sigma_{\t_n,n}
\end{array}\right.$$
where $I$ is the identity matrix and $K_{\t_n,n}$ is the classical gain matrix 
$$K_{\t_n,n} \eqdef \Sigma_{\t_n,n} (C_{\t_n,n})^T\left(C_{\t_n,n}\Sigma_{\t_n,n}(C_{\t_n,n})^T + D_{\t_n,n} \Sigma^Y_n (D_{\t_n,n})^T \right)^{-1} \eqsp.$$ 
The second step is the predicting step :
$$\left\lbrace\begin{array}{lll}
m_{\t_{n+1},n+1} &=& A_{\t_{n+1},n+1}\hat{m}_{\t_n,n}+a_{\t_{n+1},n+1}\\
\Sigma_{\t_{n+1},n+1} &=& A_{\t_{n+1},{n+1}}\hat{\Sigma}_{\t_n,n} (A_{\t_{n+1},n+1})^T+B_{\t_{n+1},n+1}\Sigma^X_{n+1}(B_{\t_{n+1},n+1})^T \eqsp.
\end{array}\right.$$
Then all the Kalman filters attached to each realization $\t^i_{n+1}$ for $i \in \intvect{1}{\Nis}$ interact through their potential $\bar{G}_{n+1}(\theta^i_{n+1},\kacx[i]{n+1})$ defined in \eqref{potential on distribution} by
\begin{align*}
\bar{G}_{n+1}(\theta^i_{n+1},\kacx[i]{n+1}) &= \kacx[i]{n+1}(G_{\t^i_{n+1},n+1})\\
 &= \normdist(m_{\t^i_{n+1},n+1},\Sigma_{\t^i_{n+1},n+1})(G_{\t^i_{n+1},n+1})
\end{align*}
where $G_{\t^i_{n+1},n+1}$ is the likelihood function defined for every $x_{n+1}\in\spx{n+1}$ by 
$$G_{\t^i_{n+1},n+1}(x_{n+1}) = \frac{\rmd\normdist(C_{\t^i_{n+1},n+1}x_{n+1},\Sigma^Y_{n+1})}{\rmd\normdist(0,\Sigma^Y_{n+1})} \eqsp.$$
One finally ends up with the following expression:
\begin{multline}\label{def:pot:IKF}
 \bar{G}_{n+1}(\theta^i_{n+1},\kacx[i]{n+1}) \\
 = \frac{\rmd\normdist(C_{\t^i_{n+1},n+1}m_{\t^i_{n+1},n+1},C_{\t^i_{n+1},n+1}\Sigma_{\t^i_{n+1},n+1}(C_{\t^i_{n+1},n+1})^T+\Sigma^Y_{n+1})}{\rmd\normdist(0,\Sigma^Y_{n+1})} \eqsp.
\end{multline}
See \cite{delmoral:2004} for further details.
The interacting Kalman filter for this general example is given by Algorithm \ref{IKF}. \\
 \begin{algorithm}[H]  \label{IKF} 
	\KwData{$\kacb{0}$, $(\transb{p})_{p=0}^n$, $(\bgb{p})_{p=0}^n$, $m_{\t^i_0,0}$ and $\Sigma_{\t^i_0,0}$ }
	\KwResult{Interacting Kalman approximation of $\kacb{n}$}
  \tcc{Initialization} 
  \For{$i \gets 1$ \KwTo ${\Nis}$}{
  Sample $\tilde{X}_0^i=(\t_0^i,\kacx[i]{0})\sim\kacti{0}$, \textit{i.e.} 
  $\t^i_0\stackrel{i.i.d}{\sim}\eta^{\Theta}_0$ and $\kacx[i]{0}=\normdist(m_{\t^i_0,0},\Sigma_{\t^i_0,0})$ \;
  }
  \For{$p \gets 0$ \KwTo $n-1$}{
  \tcc{Selection of Kalman filters}
  Sample $I_p=(I_p^i)_{i=1}^{N_1}$ according to a multinomial distribution with probability proportional to $ \left(\bar{G}_p(\theta^k_p,\kacx[k]{p})\right)_{k=1}^{N_1} $ given by \eqref{def:pot:IKF} \;
  \For{$i \gets 1$ \KwTo $\Nis$}{
  \tcc{Updating step for each Kalman filter}
  $\left\lbrace\begin{array}{lll}
\hat{m}_{\t^{I_p^i}_p,p} &=& m_{\t^{I_p^i}_p,p} + K_{\t^{I_p^i}_p,p}(Y_n-C_{\t^{I_p^i}_p,p} m_{\t^{I_p^i}_p,p})\\
\hat{\Sigma}_{\t^{I_p^i}_p,p} &=& \left(I - K_{\t^{I_p^i}_p,p} C_{\t^{I_p^i}_p,p}\right)\Sigma_{\t^{I_p^i}_p,p}
\end{array}\right.$ \;
  \tcc{Mutation of each island}
  Sample independently $\theta^i_{p+1}$ according to $M^{\Theta}_{p+1}(\theta^{I^i_p}_p,.)$ \;
  \tcc{Prediction step for each Kalman filter}
  $ \left\lbrace\begin{array}{lll}
m_{\t^i_{p+1},p+1} &=& A_{\t^i_{p+1},p+1}\hat{m}_{\t^{I_p^i}_p,p}+a_{\t^i_{p+1},p+1}\\
\Sigma_{\t^i_{p+1},p+1} &=& A_{\t^i_{p+1},{p+1}}\hat{\Sigma}_{\t^{I_p^i}_p,p} (A_{\t^i_{p+1},p+1})^T+B_{\t^i_{p+1},p+1}\Sigma^X_{p+1}(B_{\t^i_{p+1},p+1})^T
\end{array}\right.$
  }
  $p\longleftarrow p+1$ 
  }
   \caption{Interacting Kalman Filter - IKF}
  \end{algorithm}
Secondly, when the non linear \autoref{quenched non linear equation} can be solved analytically \textit{i.e.} when one has access to the exact measure $\kacx{n}$, one can apply a simple interacting particle model as described in \autoref{particle exact measure}, where each particle corresponds to the pair: parameter and exact measure.\\
However, in most cases, this equation cannot be solved analytically, so that an additional approximation is needed in order to estimate the measure $\kacx[i]{n}$ for each $i \in \intvect{1}{\Nis}$. The next subsection is dedicated to the derivation of an algorithm to deal with this constraint.

\subsection{Labeled island particle model}
To tackle the case where $\kacx[i]{n}$, $i \in \intvect{1}{\Nis}$ is not analytically known, the idea consists in using a particle estimation of $\kacx[i]{n}$ inside the previous interacting particle model. The ensuing algorithm will be called \textit{labeled island particle model} in reference to the island particle model developed in \cite{verge:dubarry:delmoral:moulines:2013}, even if in the present case, each island $i$  have a label $\theta^i_n$ whose evolution is given by the Markov kernel $\transt{n}$.
The labeled island particle model consists in associating to each term of the sequence $(\t^i_n)_{i=1}^{\Nis}$ a sub $\Nin$-interacting particle system.
We call sub $\Nin$-interacting particle system associated with the sequence $((G_{\t^i_n,n},\transx[i]{n}))_{n \in \NN}$ and the initial distribution $\kacx[i]{0}$, 
the sequence of non-homogeneous Markov chain $(\xi_n^{i,j})_{j=1}^{\Nin}$ taking value in the product space $\spx[\Nin]{n}$, that is :
$$\xi^{i,[\Nin]}_n \eqdef (\xi_n^{i,j})_{j=1}^{\Nin} \eqdef (\xi_n^{i,1},\ldots,\xi_n^{i,\Nin}) \in \spx[\Nin]{n}\eqdef\underbrace{\spx{n}\times\ldots\times\spx{n}}_{\Nin\text{ times}} \eqsp.$$
The initial state of the Markov chain $(\xi_0^{i,j})_{j=1}^{\Nin}$ consists in sampling $\Nin$ independent random variables with common distribution $\kacx[i]{0}$.\\
The interacting particle system, denoted by $(\xi_n^{i,j})_{j=1}^{\Nin}$, explore the state space $\spx{n}$ and with the dynamic given to it, empirically sample the law $\kacx[i]{n}$.\\
Denoting the empirical measure 
\begin{equation} \label{def:kacxN}
\kacxN[i]{n} \eqdef \frac{1}{\Nin}\sum\limits_{j=1}^{\Nin}\delta_{\xi_n^{i,j}} \eqsp,
\end{equation}
the elementary transition of the process $\xi_n^{i,[\Nin]}$ from $\spx[\Nin]{n}$ to $\spx[\Nin]{n+1}$ is given for any $x_n^{[\Nin]}=(x^1_n, \ldots, x^{\Nin}_n) \in \spx[\Nin]{n}$ by 
\begin{align*} 
\PP^{\Nin}_{\kacx{0}} \left(\xi_{n+1}^{i,[\Nin]} \in \rmd x_{n+1}^{[\Nin]} \, | \, \xi_{n}^{i,[\Nin]}\right) 
& \eqdef \prod_{j=1}^{\Nin} \Phi^X_{n+1} \left((\t_{n}^i ,\t_{n+1}^i),\kacxN[i]{n} \right) (\rmd x^j_{n+1}) \\
&= \prod_{j=1}^{\Nin}  \bgx[i]{n} (\kacxN[i]{n}) \transx[i]{n+1} (\xi^{i,j}_{n}, \rmd x^j_{n+1}) \quad \text{using \eqref{quenched non linear equation}}  \\
&= \prod_{j=1}^{\Nin}  \sum_{k=1}^{\Nin} \dfrac{G_{\t_{n}^i,n}(\xi^{i,k}_{n}) }{\sum_{\ell=1}^{\Nin} G_{\t_{n}^i,n}(\xi^{i,\ell}_{n}) } \transx[i]{n+1} (\xi^{i,k}_{n}, \rmd x^j_{n+1}) \quad \text{by \eqref{def:bgx}}   \eqsp.
\end{align*}
Define the mapping $\tilde{\Phi}^X_n$ by 
\begin{align*}
 \begin{array}{r c l}
  \tilde{\Phi}^X_n : \spt{n-1} \times \spt{n} \times \psp (\spx{n-1}) & \rightarrow & \psp(\spx{n}) \\
  ((u,v),\nu) & \mapsto & \prod_{j=1}^{\Nin} \Phi^X_{n}((u,v),\nu)(\rmd x^j_n) \eqsp,
  \end{array}
 \end{align*}
then
\begin{equation} \label{def:phitiX}
\kacxN[i]{n+1}  = \tilde{\Phi}^X_{n+1} \left((\t_{n}^i ,\t_{n+1}^i),\kacxN[i]{n} \right) \eqsp.
\end{equation}
So, the evolution of the particle swarm $\xi_{n}^{i,[\Nin]}$ consists in two steps: a selection and a mutation. In the selection step, the particles are selected multinomially with probability proportional to their potentials $\left(G_{\t_{n}^i,n}(\xi^{i,j}_{n})\right)_{j=1}^{\Nin}$. Then the mutation step is performed independently using the kernel $\transx[i]{n+1}$. Hence, at each iteration $n \in \NN$, the empirical measure $\kacxN[i]{n}$ approximates $\kacx[i]{n}$ when $\Nin$ tends to $\infty$. Replacing $\kacx[i]{n}$ by $\kacxN[i]{n}$
inside the first algorithm presented, one gets a nested particle model named \textit{labeled island particle model}.\\
In order to derive precisely this algorithm, first introduce the following sequence $\Xti[n]$ on $\spb{n} = \spt{n} \times \psp (\spx{n})$, defined by $\Xti[n] \eqdef (\T_n,\kacXTN{n})$, \textit{i.e.} the couple environment and empirical measure of the process $X_n$ conditionally on $\T_{0:n}$, where $\kacXTN{n} \eqdef \sum_{j=1}^{\Nin} \delta_{\xi^{j}_n} / \Nin$. \\
\begin{prpstn}
$\Xti[n]$ is a $\spb{n}$-Markov chain with transition kernel $\transti{n}$ defined for every function $\fb_n \in \bmf(\spb{n})$ and $(u,\nu) \in \spb{n}$ by
\begin{equation} \label{def:transti}
\transti{n} (\fb_n)(u,\nu) = \int_{\spt{n}} \transt{n}(u,\rmd v) \fb_n(v,\tilde{\Phi}^X_n((u,v),\nu)) \eqsp,
\end{equation}
where $\tilde{\Phi}^X_n$ is defined in \eqref{def:phitiX}, and with initial distribution $\kacti{0} \in \psp (\spb{0})$ given by
$$\kacti{0}(\rmd(u,\nu)) \eqdef \kact{0}(\rmd u) \delta_{\eta^{X,\Nin}_{\theta_0,0}}(\rmd \nu) \eqsp.$$
 \end{prpstn}
 \begin{proof}
Let $\filttilde{n}$ stands for the $\sigma$-algebra generated by the random variables $\Xti[p]$, $0 \leq p \leq n$. For all $\fb_n \in \bmf(\spb{n})$:
  \begin{align*}
   \cexp[\kacti{0}]{\fb_n(\Xti[n])&}{\filttilde{n-1}}\\
   &= \cexp[\kacti{0}]{\fb_n(\Theta_n, \eta^{X,\Nin}_{\Theta_{0:n},n})}{\filttilde{n-1}} \\
   &= \cexp[\tilde{\eta}_0]{\fb_n(\Theta_n,\tilde{\Phi}^X_n((\Theta_{n-1},\Theta_n),\eta^{X,\Nin}_{\Theta_{0:n-1},n-1})}{\filttilde{n-1}} \quad \text{by \eqref{def:phitiX}} \eqsp.
  \end{align*}
 Recalling that $\Xti[n-1]=(\Theta_{n-1},\eta^{X,\Nin}_{\Theta_{0:n-1},n-1})$, one can conclude that  
  \begin{align*}
\cexp[\kacti{0}]{\fb_n(\Xti[n])&}{\filttilde{n-1}} \\
&= \cexp[\kacti{0}]{\fb_n(\Xti[n])}{\Xti[n-1]}\\
&= \int_{\spt{n}} \fb_n(\t_n,\tilde{\Phi}^X_n((\Theta_{n-1},\theta_n),\eta^{X,\Nin}_{\Theta_{0:n-1},n-1}) \transt{n}(\Theta_{n-1},\rmd \t_n) \eqsp.
\end{align*}
 \end{proof}
To the Markov chain $\Xti[n]$, one may associate the Feynman-Kac distribution defined for every $\fb_n \in \bmf(\spb{n})$ by 
\begin{equation} \label{def:kacti}
\kacti{n}(\fb_n) \eqdef \unkacti{n}(\fb_n) / \unkacti{n}(\1) \eqsp,
\end{equation}
where $\unkacti{n}$ is defined such that 
$$\unkacti{n}(\fb_n) \eqdef \esp[\kacti{0}]{\fb_n(\Xti_n) \prod_{p=0}^{n-1} \overline{G}_p (\Xti_p)} \eqsp,$$
with $\overline{G}_p$ defined in \eqref{potential on distribution}.\\
In a similar way to $\kacb{n}$, the measure  $\kacti{n}$ satisfies a recursive equation $\kacti{n} = \bgb{n-1}(\kacti{n-1}) \transti{n},$
with $\bgb{n-1}$ the application defined in \autoref{prop:non linear equation distribution}.
This non linear equation can be rewritten as
\begin{equation} \label{eq:rec:kacti}
\kacti{n} = \tilde{\Phi}_n(\kacti{n-1}) \eqsp,
\end{equation}
where the mapping $\tilde{\Phi}_n$ is defined as follows : 
\begin{align}\label{def:phiti}
 \begin{array}{r c l}
  \tilde{\Phi}_n : \psp(\spb{n-1}) & \rightarrow & \psp(\spb{n})\\
  \eta & \mapsto & \bgb{n-1}(\eta) \transti{n} \eqsp.
 \end{array}
\end{align}
As in \autoref{sec:IKF}, when this equation cannot be solved analytically one may use a particle model to approximate the probability measure $\kacti{n}$. In this case, the particles $ \{\Xti^i_n \eqdef (\t^i_n,\kacxN[i]{n})$, $i \in \intvect{1}{\Nis} \}$, would be testing points on the state space $\spb{n}$, for $(\Nis, \Nin) \in (\NN^*)^2$.
These particles explore the state space $\spb{n}$ and their dynamics empirically sample the law $\kacti{n}$ when $\Nis$ gets large.
An interacting particle system associated with the couple $(\overline{G}_n,\transti{n})$ and the initial distribution $\kacti{0}$, is a sequence of non-homogeneous Markov chain, $\Xti[n]^{[\Nis]}$, taking value in the product space $\spb{n}^{\Nis}$, defined by
$$\Xti[n]^{[\Nis]} \eqdef (\Xti[n]^{i})_{i=1}^{\Nis} = (\Xti[n]^{1},\ldots,\Xti[n]^{\Nis}) \in \spb{n}^{\Nis}\eqsp.$$
The initial state of the Markov chain $\Xti[0]^{[\Nis]}$ consists in $\Nis$ independent random variables with common distribution $\kacti{0}$.
Denote by $\kacti{n}^{\Nis}$ the empirical measure at time $n$, which is defined by 
\begin{equation}\label{def:kactiN}
\kacti{n}^{\Nis}\eqdef \frac{1}{\Nis}\sum\limits_{i=1}^{\Nis} \delta_{\Xti[n]^i}.
\end{equation}
The elementary probability transition, is given for any $\xb[n+1]^{[\Nis]} \in \spb{n+1}^{\Nis}$ by 
$$ \prob^{\Nis}_{\kacti{0}}(\Xti[n+1]^{[\Nis]} \in \rmd \xb[n+1]^{[\Nis]} \, | \, \Xti[n]^{[\Nis]} ) 
= \prod_{i=1}^{\Nis} \bgb{n}(\kacti{n}^{\Nis}) \transti{n+1} (\Xti[n]^i, \rmd \xb[n+1]^i) \eqsp. $$
The particle evolution is summarized in \autoref{modified particle} where by definitions \eqref{potential on distribution} and \eqref{def:kacxN},
\begin{align*}
 \overline{G}_{n}(\Xti[n]^i) = \frac{1}{\Nin}\sum_{j=1}^{\Nin}G_n(\t^i_{n},\xi^{i,j}_{n})
 =  \frac{1}{\Nin}\sum_{j=1}^{\Nin}G_{\t^i_{n},n}(\xi^{i,j}_{n}) \eqsp.
\end{align*}
The ensuing algorithm is described in Algorithm \ref{algorithme island}. \\
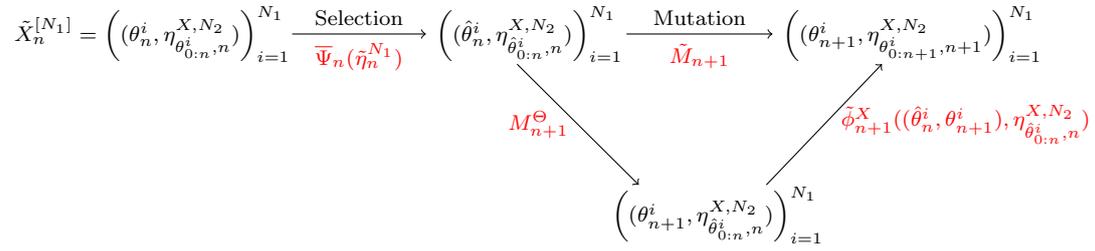
\begin{figure}[!h]
   \begin{tikzpicture}[scale=0.8]
 \draw (0,0) node {\footnotesize{$\Xti[n]^{[\Nis]} = \left((\t^{i}_{n},\kacxN[i]{n})\right)_{i=1}^{\Nis}$}};
 \draw (6.2,0) node {\footnotesize{$\left((\hat{\t}^{i}_{n},\hatkacxN[i]{n})\right)_{i=1}^{\Nis}$}};
 \draw (12.5,0) node {\footnotesize{$\left((\t^{i}_{n+1},\kacxN[i]{n+1})\right)_{i=1}^{\Nis}$}};
 \draw [->] (2.3,0) -- (4.5,0) node [midway,above] {\footnotesize{Selection}} node[midway, below] {\footnotesize{\color{red}$\bgb{n}(\kacti{n}^{\Nis})$}};
 \draw [->] (7.8,0) -- (10.2,0) node [midway, above] {\footnotesize{Mutation}}  node[midway, below] {\footnotesize{\color{red}$\transti{n+1}$}};
 \draw (9.3,-3) node {\footnotesize{$\left((\t^{i}_{n+1},\hatkacxN[i]{n})\right)_{i=1}^{\Nis}$}};
 \draw [->] (6,-0.5) -- (8,-2.5) node [midway, left] {\footnotesize{{\color{red}$\transt{n+1}$}}};
 \draw [->] (10.1,-2.5) -- (12,-0.5) node [midway, right] {\footnotesize{ {\color{red}$\tilde{\phi}^X_{n+1}((\hat{\t}^{i}_{n},\t^{i}_{n+1}),\hatkacxN[i]{n}) $}}} ;
 \end{tikzpicture}
\caption{Evolution scheme of the labeled island particle model.}
\label{modified particle}
\end{figure}
~~\\
\begin{algorithm}[H] \label{algorithme island}
\KwData{$\tilde{\eta}_0$, $(\transti{p})_{p=0}^n$ and $(\bgb{p})_{p=0}^n$}
\KwResult{Particle approximation of $\tilde{\eta}_n$}
  \tcc{Initialization} 
  \For{$i \gets 1$ \KwTo $N_1$}{
  Sample $\tilde{X}^i_0=(\t_0^i,\kacxN[i]{0})\sim\kacti{0}$, that is \\
  \begin{center}
  $\t^i_0\stackrel{i.i.d}{\sim}\eta^{\Theta}_0$, $\xi^{i,j}_0\stackrel{i.i.d}{\sim}\kacx[i]{0}$ and $\kacxN[i]{0}=\frac{1}{\Nin}\sum\limits_{j=1}^{\Nin}\delta_{\xi^{i,j}_0}$.\\
  \end{center}
  }
  \For{$p \gets 0$ \KwTo $n-1$}{
  \tcc{Selection of islands} 
  Sample $I_p=(I_p^i)_{i=1}^{N_1}$ according to a multinomial distribution with probability proportional to $\left(\frac{1}{N_2}\sum\limits_{j=1}^{N_2}G_p(\theta^i_p,\xi^{i,j}_p)\right)_{i=1}^{N_1} $\; 
  \For{$i \gets 1$ \KwTo $N_1$}{
  \tcc{Selection of particles inside each island} 
  Sample $J^i_p=(J^{i,j}_p)_{j=1}^{N_2}$ according to a multinomial distribution with probability proportional to $\left(G_p(\theta_p^{I^i_p},\xi^{I^i_p,j}_p)\right)_{j=1}^{N_2} $ \;
  \tcc{Mutation of each island}
  Sample independently $\theta^i_{p+1}$ according to $M^{\Theta}_{p+1}(\theta^{I^i_p}_p,.)$ \;
  \For{$j \gets 1$ \KwTo $N_2$}{
  \tcc{Mutation of particles}
  Sample $\xi^{i,j}_{p+1}$ according to $\transx[i]{p+1} (\xi^{I^i_p,J^{i,j}_p}_p,.)$ \;
  }
  } 
  $p\longleftarrow p+1$ 
  }  
  \caption{Labeled island particle algorithm}
  \end{algorithm}
For every $n\geq 0$,  $\kacti{n}^{\Nis}$ is an estimator of $\kacti{n}$, obtained through the {\lip} model, i.e. for every $\fb_n \in\bmf(\spb{n})$,
$$\kacti{n}^{\Nis} (\fb_n) = \frac{1}{\Nis} \sum_{i=1}^{\Nis} \fb_n (\t^i_n,\kacxN[i]{n})$$
converges to $\kacti{n} (\fb_n)$ when $\Nis \rightarrow +\infty$.


\section{$\Lp$ bounds } \label{sec:lpbounds}
We are interested in this section in the $\Lp$ bounds of the difference between the estimator $\kacti{n}^{\Nis}$ and the measure $\kacb{n}$. To get these bounds we will use several notations. We define them before going further. 
\subsection{Notations}
Let $(\stsp, \stfd)$ be a measurable space.
For a real-valued measurable function $h \in \bmf(\stsp)$, we denote the oscillator norm $\osc{h} \eqdef \sup_{(x, x') \in \stsp^2} |h(x) - h(x')|$, and $\OscOne(\stsp)$ the convex set of $\stfd$-measurable functions with oscillations less than one. 
The sup norm of $h$ is noted $\supn[txt]{h} \eqdef \sup_{x \in \stsp} |h(x)|$ and the $\Lp$-norm $\|.\|_{p}$.
$\bmfone(\stsp) \subset \bmf(\stsp) $ refers to the set of functions whose sup norm is less than one.
For two probability measures $(\mu,\eta) \in \mathcal{P}(\stsp)^2$, 
the Zolotarev semi-norm $\|.\|_{\Ffr}$ attached to $\Ffr$ a countable collection of bounded measurable functions in $\bmfone(\stsp)$  is defined by 
$$\|\mu - \eta\|_{\Ffr} \eqdef \sup\limits_{f \in \Ffr} |\mu(f)-\eta(f)| \eqsp.$$ 
To measure the size of a given class $\Ffr$, one considers the covering numbers \\
$\mathcal{N}(\varepsilon,\Ffr,\Lp(\mu))$
defined as the minimal number of $\Lp(\mu)$-balls of radius $\varepsilon>0$ needed to cover $\Ffr$. Let $\mathcal{N}(\varepsilon,\Ffr)$ and $I(\Ffr)$ denote respectively the uniform covering numbers and entropy integral given by 
\begin{equation}\label{def:covnum}
 \mathcal{N}(\varepsilon,\Ffr)\eqdef
 \sup_{\mu\in\mathcal{P}(\stsp)} \mathcal{N}(\varepsilon,\Ffr,\mathbb{L}^2(\mu))
 \end{equation}
\begin{equation}\label{def:entropy}
 I(\Ffr) \eqdef \int_{0}^1\sqrt{\log(1+\mathcal{N}(\varepsilon,\Ffr))} \rmd\varepsilon \eqsp.
\end{equation}
Let $\wedge$ denote the minimum operator and $\vee$ denote the maximum operator. 
For a kernel $M$ defined on $\stsp$, the Dobrushin coefficient of $M$ is 
$$
 \beta(M) \eqdef \sup_{f \in \OscOne(\stsp)}{\osc{M(f)}} \eqsp.
$$ 
Let $(d(n))_{n\geq 0}$ be a sequence defined for every $m \geq 0$ by
\begin{equation} \label{def:d}
 \left\lbrace\begin{array}{l}
  d(2m)^{2m} \eqdef (2m)_m 2^{-m} \nonumber\\
  d(2m+1)^{2m+1} \eqdef \frac{(2m+1)_{m+1}}{\sqrt{m+1/2}} 2^{-m+1/2} \eqsp,\\
 \end{array}\right.
\end{equation}
 where for any positive integers $(p,q) \in (\nsetpos)^2$, $(q+p)_p \eqdef (q+p)!/q! \eqsp.$\\
For $n \in \NN$, introduce the Feynman-Kac semi-groups $\overline{Q}_{n}$ (resp. $Q^X_{\t_{n-1:n},n}$) such that for all 
$(\xb_n, \xb[n+1]) \in \spb{n} \times \spb{n+1}$ 
(resp. $(x_n, x_{n+1}) \in \spx{n} \times \spx{n+1}$),
$$ \overline{Q}_{n+1}(\xb[n],\rmd \xb[n+1]) \eqdef \overline{G}_{n}(\xb[n]) \transb{n+1}(\xb[n],\rmd \xb[n+1]) \eqsp, $$
\begin{center}
$\left( \text{resp.} \, Q^X_{\t_{n:n+1},n+1}(x_{n},\rmd x_{n+1}) \eqdef G_{\t_{n},n}(x_{n}) \transx{n+1}(x_{n},\rmd x_{n+1})\right)$.
\end{center}
For every $(p,n) \in (\NN)^2$ such that $p<n$, set
$$ \overline{Q}_{p,n} \eqdef \overline{Q}_{p+1} \ldots \overline{Q}_n \eqsp, \quad \text{and} \quad \overline{P}_{p,n} \eqdef\overline{Q}_{p,n} / \overline{Q}_{p,n}(\1) \eqsp, $$
\begin{center}
$\left(\text{resp.} \, Q^X_{\t_{p:n},p,n} \eqdef Q^X_{\t_{p:p+1},p+1}\ldots Q^X_{\t_{n-1:n},n} \, \text{and} \, P^X_{\t_{p:n},p,n}(f_n) \eqdef  Q^X_{\theta_{p:n},p,n}(f_n) / Q^X_{\t_{p:n},p,n}(\1) \right)\eqsp,$\end{center}
and set the normalizing constant
$$\overline{G}_{p,n} \eqdef \overline{Q}_{p,n}(\1) \eqsp, \quad  (\text{resp.}  G_{\t_{p:n},p,n} \eqdef Q^X_{\t_{p:n},p,n}(\1)) \eqsp.$$
Finally, set
$$ \overline{g}_{p,n} \eqdef \sup\limits_{(\xb[p],\overline{y}_p) \in (\spb{p})^2} \dfrac{\overline{G}_{p,n}(\overline{x}_p)}{\overline{G}_{p,n}(\overline{y}_p)} \eqsp, \quad
\left(\text{resp.}~ g_{\t_{p:n},p,n} \eqdef \sup\limits_{(x_p,y_p) \in (\spx{p})^2} \dfrac{G_{\t_{p:n},p,n}(x_p)}{G_{\t_{p:n},p,n}(y_p)} \right)\eqsp.$$

In order to study the difference between $\kacti{n}^{\Nis}$ and $\kacb{n}$, we use several results taken from \cite{delmoral:2004}. Then, we will always assume that for all $n \in \NN$, the potential functions $G_{\t_n,n}$ defined in \eqref{eq:qpot} satisfy the following condition ($G_\t$): \\
there exists a sequence of strictly positive number $\epsilon_n(G_\t)\in (0,1]$ such that for any $(x_n,y_n) \in (\spx{n})^2$ :
\begin{equation}\label{condition:Gt}
G_{\t_n,n}(x_n) \geq \epsilon_n (G_\t)G_{\t_n,n}(y_n)>0 \tag{$G_\t$}
\end{equation}
Therefore, for all $n \in \NN$, the potential functions $\overline{G}_n$ satisfy the following condition ($\overline{G}$):\\
there exists a sequence of strictly positive number $\epsilon_n(\overline{G})\in (0,1]$ such that for any $(\xb[n],\overline{y}_n) \in (\spb{n})^2$ :
\begin{equation}\label{condition:Gb}
\overline{G}_n(\xb[n]) \geq \epsilon_n (\overline{G}) \overline{G}_n(\overline{y}_n)>0 \tag{$\overline{G}$}
\end{equation}
Moreover we always assume that the collection of distributions $\left(\transb{n+1}(\xb[n],.)\right)_{\xb[n]\in \spb{n}}$ are absolutely continuous with one another. That is for every $n\geq 0$ and $(\xb[n], \overline{y}_n) \in (\spb{n})^2$, one has  
$$\transb{n+1}(\xb[n],.)\ll \transb{n+1}(\overline{y}_n,.) \eqsp.$$
In addition, we assume that the collection of distributions $\left(\transx{n+1} (x_{n},.)\right)_{x_n\in\spx{n}}$ are absolutely continuous with one another. That is for every $n\geq 0$, $\theta_{n+1}\in \spt{n+1}$ and $(x_n, y_n) \in (\spx{n})^2$, one has : 
$$\transx{n+1}(x_n,.)\ll \transx{n+1}(y_n,.) \eqsp.$$
Note that for time homogeneous models on finite spaces condition those conditions are met as soon as the Markov chain is aperiodic and irreducible. Some examples are illustrated by typical examples in \cite{delmoral:2004}.

\subsection{$\Lp$ bound}
Consider that for all $n \in \NN$, the product space $\spb{n} = \spt{n} \times \psp(\spx{n})$ is equipped with the norm $\|\cdot\|_{\spb{n}}$ such that for all $(u,v) \in (\spt{n})^2$ and $(\nu, \eta) \in (\psp(\spx{n}))^2$,
$$
\| (u,\eta)-(v,\nu) \|_{\spb{n}} = |u-v| + \|\eta - \nu\|_{\Ffr_{n}} \eqsp.
$$
where $\Ffr_{n}$ is a countable collection of functions in $\bmfone(\spx{n})$.
\begin{thrm}\label{theo:Lp bound}
For any $p \in \nsetpos$, $n \in \NN$, let $\fb_n \in \OscOne(\spb{n})$ be a $k_n$-Lipschitz function.
Assume that for any $\t_n \in \spt{n}$, the kernel transition $\transx{n}$ can be written as $\transx{n} (x_{n-1},\rmd x_n) = m^X_{\t_n,n}(x_{n-1},x_n) p_{\t_n,n}(\rmd x_n)$ for some measurable function $m^X_{\t_n,n}$ on $\spx{n-1} \times \spx{n}$ and some probability measure $p_{\t_n,n} \in \psp(\spx{n})$. 
Furthermore, assume that there exists a collection of mappings $\alpha_{\t_n,n}$ on $\spx{n}$ such that 
\begin{center}
$\sup_{x_{n-1} \in \spx{n-1}} | \log m^X_{\t_n,n}(x_{n-1},x_{n})| \leq \alpha_{\t_n,n}(x_{n})$ 
\end{center}
with $p_{\t_n,n}(e^{3\alpha_{\t_n,n}}) < \infty$. \\
Then, the $\Lp$ error is bounded by
\begin{align}
 \|\kacti{n}^{N_1}(\fb_n) - \kacb{n}(\fb_n)\|_p &\leq k_n \frac{a(p)}{\sqrt{\Nin}} (I(\Ffr_{n})+b(n)) + 2 \frac{d(p)}{\sqrt{\Nis}} \sum_{q=0}^n \overline{g}_{q,n}\beta(\overline{P}_{q,n}) \eqsp,
\end{align}
where the sequence $d(n)$ is defined in \eqref{def:d}, $I(\Ffr_{n})$ is defined in \eqref{def:entropy}, $(b(n))_{n\geq 0}$ is defined by 
\begin{center} \label{def:b}
$b(0)=0$ and $b(n+1)\leq g_{\theta_n,n} p_{\theta_{n+1},n+1}(e^{3\alpha_{\theta_{n+1},n+1}})\sum\limits_{q=0}^n g_{\theta_{q:n},q,n}\beta(P_{\theta_{q:n},q,n}) \eqsp,$
\end{center}
and $a(n)$ is a sequence such that for all $n\in \nsetpos$, $a(n)\leq c\left[n/2\right]!$ with $c$ a universal constant.
\end{thrm}

\begin{proof}
Let $\fb_n \in \OscOne(\spb{n})$ be a $k_n$-Lipschitz function, and apply triangular inequality:
 $$\| \kacti{n}^{\Nis}(\fb_n) - \kacb{n}(\fb_n) \|_p 
 \leq \|\kacti{n}^{\Nis}(\fb_n) - \kacb{n}^{\Nis}(\fb_n)\|_p 
 + \| \kacb{n}^{\Nis}(\fb_n) - \kacb{n}(\fb_n) \|_p \eqsp,$$
 where $\kacb{n}$ is defined in \eqref{Feynman-Kac flow in distribution}.
Then using Theorem 7.4.4 from \cite{delmoral:2004},
one can bound the second term 
$$\|\kacb{n}^{\Nis}(\fb_n) - \kacb{n}(\fb_n)\|_p \leq 2 \frac{d(p)}{\sqrt{\Nis}} \sum_{q=0}^n \overline{g}_{q,n}\beta(\overline{P}_{q,n}) \eqsp.$$
Therefore, in order to bound the first term, use the definitions of $\kacti{n}^{\Nis}$ and $\kacb{n}^{\Nis}$ in \eqref{def:kactiN} and \eqref{def:kacb} respectively :  \begin{align*}
\|\kacti{n}^{\Nis}(\fb_n) - \kacb{n}^{\Nis}(\fb_n)\|_p 
  &= \esp[\kacb{0}]{ |\kacti{n}^{\Nis}(\fb_n) - \kacb{n}^{\Nis} (\fb_n) |^p }^{1/p}\\
  &= \esp[\kacb{0}]{ \left| \frac{1}{\Nis} \sum_{i=1}^{\Nis} \left\{\fb_n(\t^i_n,\kacxN[i]{n}) - \fb_n(\t^i_n,\kacx[i]{n}) \right\} \right|^p }^{1/p}\\
  &\leq \esp[\kacb{0}]{ \left( \frac{1}{\Nis} \sum_{i=1}^{\Nis} | \fb_n(\t^i_n,\kacxN[i]{n}) - \fb_n (\t^i_n, \kacx[i]{n}) | \right)^p }^{1/p}
 \end{align*}
As $\fb_n$ is $k_n$-Lipschitz, it follows
\begin{align*}
 \left|\fb_n(\t^i_n,\kacxN[i]{n}) - \fb_n(\t^i_n, \kacx[i]{n}) \right| 
 & \leq k_n \left\| (\t^i_n,\kacxN[i]{n}) - (\t^i_n, \kacx[i]{n}) \right\|_{\spb{n}} \\
 &\leq k_n \left\| \kacxN[i]{n} - \kacx[i]{n} \right\|_{\Ffr_{n}} \eqsp.
\end{align*}
where $\Ffr_{n}$ is a countable collection of functions in $\bmfone(\spx{n})$
Therefore one gets 
\begin{align*}
 \|\kacti{n}^{\Nis}(\fb_n) - \kacb{n}^{\Nis}(\fb_n) \|_p 
 \leq \esp[\kacb{0}]{ \left(\frac{1}{\Nis} \sum_{i=1}^{\Nis} k_n \left\| \kacxN[i]{n} - \kacx[i]{n} \right\|_{\Ffr_{n}} \right)^p }^{1/p} \eqsp.
\end{align*}
Denoting by $\theta^*$ the value at which the maximum of the Zolotarev semi-norms for $i \in \intvect{1}{\Nis}$ is reached, yields to 
\begin{align*}
 \|\kacti{n}^{\Nis}(\fb_n) - \kacb{n}^{\Nis}(\fb_n)\|_p 
 &\leq k_n \esp[\kacb{0}]{ \|\eta^{X,\Nin}_{\theta^*_{0:n},n} - \eta^X_{\theta^*_{0:n},n}\|^p_{\Ffr_{n}} }^{1/p} \eqsp.
 \end{align*}
But, according to [\cite{delmoral:2004}, Corollary 7.4.4] (p. 247),
 \begin{align*}
  \esp[\kacb{0}]{ \|\eta^{X,\Nin}_{\theta^*_{0:n},n} - \eta^X_{\theta^*_{0:n},n}\|^p_{\Ffr_{n}} }^{1/p} 
  &\leq \frac{a(p)}{\sqrt{\Nin}}(I(\Ffr_{n})+b(n)) \eqsp,
 \end{align*}
  where $I(\Ffr_{n})$ is the entropy of $\Ffr_{n}$ defined in \eqref{def:entropy}.
Finally one ends up with 
$$\|\kacti{n}^{\Nis}(\fb_n) - \kacb{n}^{\Nis}(\fb_n)\|_p 
\leq k_n \frac{a(p)}{\sqrt{\Nin}} (I(\Ffr_{n})+b(n)) 
+ 2 \frac{d(p)}{\sqrt{\Nis}} \sum_{q=0}^n \overline{g}_{q,n}\beta(\overline{P}_{q,n}) \eqsp.$$
\end{proof}

\subsection{Time uniform bound}
Before stating the uniform estimate we define the following two additional conditions. \\
There exists some integer $\overline{m}\geq 1$ and some numbers $\varepsilon_n(\overline{M})\in\left]0,1\right[$ such that for $n \in \NN$ and $(\xb_n,\overline{y}_n) \in \spb{n}^2 $, one has :
\begin{equation}\label{condition:Mb}
\overline{M}_{n,n+\overline{m}}(\xb_n,.) \eqdef \overline{M}_{n+1} \ldots \overline{M}_{n+\overline{m}}(\overline{x}_n,.) \geq \varepsilon_n(\overline{M})\overline{M}_{n,n+\overline{m}}(\overline{y}_n,.) \tag{$(\overline{M})_{\overline{m}}$} 
\end{equation}
For all $i\in\intvect{1}{\Nis}$, there exists some integer $m_i\geq 1$ and some numbers $\varepsilon_n(M^X_{\theta^i})\in\left]0,1\right[$ such that for $n \in \NN$, $(x_n,y_n) \in (\spx{n})^2$, and $\t^i_{n+1:n+m_i} \in \spt{n+1}\times\ldots\times\spt{n+m}$, one has :
\begin{multline}  \label{condition:Mt}
M^X_{\theta^i_{n+1:n+m_i},n,n+m_i}(x_n,.) \eqdef M^X_{\theta^i_{n+1},n+1}\ldots M^X_{\theta^i_{n+m_i},n+m_i}(x_n,.) \\
\geq \varepsilon_n(M^X_{\theta^i})M^X_{\theta^i_{n+1:n+m_i},n,n+m}(y_n,.) \tag{$(M^X_{\theta^i})_{m_i}$} 
\end{multline}

\begin{thrm}\label{theorem:lp:uniform}
 Suppose that conditions \eqref{condition:Gb}, \eqref{condition:Mb} are met for some integer $\overline{m} \geq 1$ and some pair parameters $(\epsilon_n(\overline{G}),\varepsilon_n(\overline{M}))$ and set $\epsilon(\overline{G}) \eqdef \wedge_{n\geq 0} \epsilon_n(\overline{G})$ and $\varepsilon(\overline{M}) \eqdef \wedge_{n\geq 0}\varepsilon_n(\overline{M})$. \\
 Moreover, assume that for all $i \in \intvect{1}{\Nis}$ conditions \eqref{condition:Gt} and \eqref{condition:Mt} hold true for some sequence of integer $m_i$ and some pair parameters $(\epsilon_n(G_{\theta^i}),\varepsilon_n(M^X_{\theta^i}))$ and set $\epsilon(G_{\theta^i}) \eqdef \wedge_{n\geq 0} \epsilon_n(G_{\theta^i})$ and $\varepsilon(M^X_{\theta^i}) \eqdef \wedge_{n\geq 0}\varepsilon_n(M^X_{\theta^i})$. Set $m\eqdef\vee_{i}m_i$. \\
Further assume that for all $n\geq 0$ and $\theta^i_n\in \spt{n}$ the kernel transition $M^X_{\theta^i_n,n}$ has the form $M^X_{\theta^i_n,n}(x_{n-1},\rmd x_n)=m^X_{\t^i_n,n}(x_{n-1},x_n)p_{\t^i_n,n}(\rmd x_n)$ for some measurable function $m^X_{\theta^i_n,n}$ on $\spx{n-1} \times \spx{n}$ and some probability measure $p_{\theta^i_n,n} \in \psp(\spx{n})$. \\
Also assume that $\sup_{x_{n-1} \in \spx{n-1}}|\log m^X_{\t^i_n,n}(x_{n-1},x_n)|\leq \alpha_{\theta^i_n,n}(x_n)$ with $p_{\t^i_n,n}(e^{3\alpha_{\t^i_n,n}})<\infty$ for some collection of mappings $\alpha_{\t^i_n,n}$ on $E^X_n$, and set :
  $$p_{\t^i}(e^{3\alpha_{\t^i}}) \eqdef \sup\limits_{n\geq 0}p_{\t^i_n,n}(e^{3\alpha_{\t^i_n,n}})<\infty\quad \text{and} \quad p_{\t}(e^{3\alpha_{\t}}) \eqdef \vee_{i}p_{\t^i}(e^{3\alpha_{\t^i}}) \eqsp.$$
Then for any $p \in \nsetpos$, any $k_n$-Lipschitz functions $\fb_n \in \OscOne(\spb{n})$ one has :
$$\sup\limits_{n\geq 0} \, \sup\limits_{\fb_n\in \OscOne(\spb{n})} \|\kacti{n}^{\Nis}(\fb_n)-\kacb{n}(\fb_n)\|_p \leq \frac{2d(p)\overline{m}}{\sqrt{\Nis}\varepsilon(\overline{M})^3\epsilon(\overline{G})^{2\overline{m}-1}}+ \frac{k \,a(p)}{\sqrt{\Nin}}\left(I+\frac{mp_{\t}(e^{3\alpha_{\theta}})}{\varepsilon(M^X_{\theta})^3\epsilon(G_{\theta})^{2m}}\right)$$
with $$k=\sup_n k_n \eqsp, \quad \varepsilon(M^X_{\theta})=\wedge_{i} \varepsilon(M^X_{\theta^i})>0\eqsp, \quad\epsilon(G_{\theta})= \wedge_{i} \epsilon(G_{\theta^i})>0 \eqsp, \quad I \eqdef \sup\limits_{n\geq 0} I(\Ffr_{n})<\infty \eqsp.$$
\end{thrm}
\begin{proof}
 \begin{multline*}
  \sup\limits_{n\geq 0} \, \sup\limits_{\fb_n\in \OscOne(\spb{n})} \|\kacti{n}^{\Nis}(\fb_n)-\kacb{n}(\fb_n)\|_p \\
  \leq \sup\limits_{n\geq 0} \, \sup\limits_{\fb_n\in \OscOne(\spb{n})} 
  \left(\|\kacti{n}^{\Nis}(\fb_n)-\kacb{n}^{\Nis}(\fb_n)\|_p 
  + \|\kacb{n}^{\Nis}(\fb_n)-\kacb{n}(\fb_n)\|_p \right)\\
  \leq \sup\limits_{n\geq 0} \, \sup\limits_{\fb_n\in \OscOne(\spb{n})}\|\kacti{n}^{\Nis}(\fb_n)-\kacb{n}^{\Nis}(\fb_n)\|_p 
  +\sup\limits_{n\geq 0} \, \sup\limits_{\fb_n\in \OscOne(\spb{n})}\|\kacb{n}^{\Nis}(\fb_n)-\kacb{n}(\fb_n)\|_p
 \end{multline*}
From [\cite{delmoral:2004}, Theorem 7.4.4] (p. 247), one has 
$$\sup\limits_{n\geq 0} \, \sup\limits_{\fb_n\in \OscOne(\spb{n})}\|\kacb{n}^{\Nis}(\fb_n)-\kacb{n}(\fb_n)\|_p\leq \frac{2d(p)\overline{m}}{\sqrt{\Nis}\varepsilon(\overline{M})^3 \epsilon(\overline{G})^{2\overline{m}-1}} \eqsp,$$
since conditions \eqref{condition:Gb} and \eqref{condition:Mb} hold true.
Then it follows that the only term one has to work on is the following.
\begin{align*}
 &\sup\limits_{n\geq 0} \, \sup\limits_{\fb_n\in \OscOne(\spb{n})}\|\kacti{n}^{\Nis}(\fb_n)-\kacb{n}^{\Nis}(\fb_n)\|_p \\
 &=\sup\limits_{n\geq 0} \, \sup\limits_{\fb_n\in \OscOne(\spb{n})}
 \left\|\frac{1}{\Nis}\sum\limits_{i=1}^{\Nis}\left\lbrace\fb_n(\theta^i_n,\kacxN[i]{n})-\fb_n(\theta^i_n,\kacx[i]{n})\right\rbrace \right\|_p \\
 &\leq \frac{1}{\Nis}\sum\limits_{i=1}^{\Nis} \sup\limits_{n\geq 0} \, \sup\limits_{\fb_n\in \OscOne(\spb{n})}
 \left\|\fb_n(\theta^i_n,\kacxN[i]{n})-\fb_n(\theta^i_n,\kacx[i]{n}) \right\|_p
\end{align*}
As the function $\fb_n$ is $k_n$-Lipschitz, for all $i \in \intvect{1}{\Nis}$,
\begin{align*}
 &\sup\limits_{n\geq 0} \, \sup\limits_{\fb_n\in \OscOne(\spb{n})}
 \left\|\fb_n(\t^i_n,\kacxN[i]{n}) - \fb_n(\t^i_n,\kacx[i]{n}) \right\|_p \\
 &=\sup\limits_{n\geq 0} \, \sup\limits_{\fb_n\in \OscOne(\spb{n})} 
 \esp[\kacb{0}]{\left|\fb_n(\t^i_n,\kacxN[i]{n}) - \fb_n(\t^i_n,\kacx[i]{n}) \right|^p }^{1/p} \\
&\leq \sup\limits_{n\geq 0} 
\esp[\kacb{0}]{k_n^p \left\|\kacxN[i]{n}-\kacx[i]{n} \right\|_{\Ffr_{n}}^p}^{1/p}\\
&\leq \sup\limits_{n\geq 0} k_n \esp[\kacb{0}]{ \left\|\kacxN[i]{n}-\kacx[i]{n} \right\|_{\Ffr_{n}}^p}^{1/p} \eqsp.
\end{align*}
Set $k \eqdef \sup_n k_n$, then
\begin{align*}
\sup\limits_{n\geq 0} \, \sup\limits_{\fb_n\in \OscOne(\spb{n})}
\left\|\fb_n(\t^i_n,\kacxN[i]{n}) - \fb_n(\t^i_n,\kacx[i]{n}) \right\|_p
&\leq k \sup\limits_{n\geq 0} \esp[\kacb{0}]{ \left\|\kacxN[i]{n}-\kacx[i]{n} \right\|_{\Ffr_{n}}^p}^{1/p} 
\end{align*}
From [\cite{delmoral:2004}, Corollary 7.4.5] (p. 249), as one assumes that there exists $m_i \geq 1$ for \\
$\t^i_{n,n+m_i}\in\spt{n}\times\ldots\times\spt{n+m_i}$,
\begin{align*}
 \sup\limits_{n\geq 0} \esp[\kacb{0}]{ \left\|\kacxN[i]{n}-\kacx[i]{n} \right\|_{\Ffr_{n}}^p}^{1/p} \leq \frac{a(p)}{\sqrt{\Nin}}\left(I +\frac{ m_i p_{\theta^i}(e^{3\alpha_{\theta^i}})}{\varepsilon(M^X_{\theta^i})^3\epsilon(G_{\theta^i})^{2m_i}}\right) \eqsp,
\end{align*}
where $I \eqdef \sup\limits_{n\geq 0} I(\Ffr_{n})<\infty$. 
One concludes easily.
\end{proof}


\section{Asymptotic analysis of the labeled island particle algorithm}
\label{sec:asymptotic}
This section deals with the asymptotic behavior of the labeled island particle algorithm. 
Especially, we focus on the almost sure convergence.

Using \autoref{theo:Lp bound} obtained in \autoref{sec:lpbounds}, one can easily get the almost sure convergence of the double estimator $\kacti{n}^{\Nis}$ toward $\kacb{n}$ under the same assumptions as in \autoref{theo:Lp bound}.
\begin{thrm}
\label{thm:asconv}
 Under the same assumptions as in \autoref{theo:Lp bound}, for all $ n\geq 0$ and for every $k_n$- Lipschitz function $\fb_n \in \OscOne(\spb{n})$, one has 
 $$\kacti{n}^{\Nis} (\fb_n) \stackrel{a.s}{\longrightarrow} \kacb{n}(\fb_n)\text{, as }N\rightarrow\infty \eqsp,$$
 with $N=\Nis\Nin$ such that $\Nis=N^{\alpha}$ and $\Nin=N^{1-\alpha}$ for all $\alpha\in\left]0,1\right[$.
\end{thrm}
\begin{proof}
Let $\fb_n \in \OscOne(\spc{n})$ be a $k_n$-Lipschitz function and $\varepsilon >0$ a real constant. For all $p \in \nsetpos$, by Markov's inequality, one has
$$
 \prob \left(|\kacti{n}^{\Nis}(\fb_n) - \kacb{n}(\fb_n) | > \varepsilon \right) 
 \leq \frac{\esp[\kacb{0}]{|\kacti{n}^{\Nis}(\fb_n) - \kacb{n}(\fb_n) |^p}}{\varepsilon^p} \eqsp.
 $$
Then, applying \autoref{theo:Lp bound}, and noting 
\begin{center}
$C(p,n) \eqdef k_n a(p) (I(\mathcal{F}_n)+b(n))$ \quad and \quad
$\tilde{C}(p,n) \eqdef 2 d(p) \sum_{q=0}^n \overline{g}_{q,n}\beta(\overline{P}_{q,n})\eqsp,$ 
\end{center}
one has 
\begin{align*}
 \|\kacti{n}^{\Nis} (\fb_n) -\kacb{n}(\fb_n)\|_p^p &\leq \left(\frac{C(p,n)}{\sqrt{N^\alpha}}+\frac{\tilde{C}(p,n)}{\sqrt{N^{1-\alpha}}}\right)^p\\
 &= \sum_{k=0}^p \binom{p}{k} \frac{C(p,n)^k \tilde{C}(p,n)^{p-k}}{N^{(\a-\frac{1}{2})k + \frac{1-\a}{2}p}} \eqsp.
\end{align*}
The finite sequence $(s_{\a,p}(k))_{k=0}^p$ defined by $s_{\a,p}(k)=(\a-1/2)k + (1-\a)p/2$ is bounded from below by
$$
m_{\a,p} \eqdef \frac{\a p}{2} \1_{0 < \a \leq 0.5} + \frac{(1-\a)p}{2} \1_{0.5 < \a < 1} \eqsp,
$$
so that
$$
 \|\kacti{n}^{\Nis} (\fb_n) -\kacb{n}(\fb_n)\|_p^p 
\leq  \dfrac{\left(C(p,n) + \tilde{C}(p,n) \right)^p}{N^{m_{\a,p}}} \eqsp.
$$
Choose $p$ a positive integer such that $m_{\a,p} \geq 2 $ \text{i.e.} satisfying
\begin{align}
 \left\{ 
 \begin{array}{l}p>\frac{4}{\alpha} \quad \text{if}\quad 0< \a \leq 0.5\\ 
 p>\frac{4}{1-\alpha} \quad\text{if}\quad 0.5 < \a <1 \eqsp.
 \end{array}\right.
\end{align}
Hence,
$$
 \|\kacti{n}^{\Nis} (\fb_n) -\kacb{n}(\fb_n)\|_p^p 
\leq  \dfrac{\left(C(p,n) + \tilde{C}(p,n) \right)^p}{N^2} \eqsp.
$$
By comparison of series of non-negative general term with a convergent Riemann series, one concludes that the series
$$\sum_{N\geq 0} \prob \left(|\kacti{n}^{\Nis} (\fb_n) -\kacb{n}(\fb_n)|>\varepsilon\right)\text{ is convergent,}$$
which implies by Borel-Cantelli's lemma, that 
$$\kacti{n}^{\Nis} (\fb_n)\stackrel{a.s}{\longrightarrow} \kacb{n}(\fb_n)\text{, as }N\rightarrow\infty \eqsp.$$
\end{proof}

Taking such kind of $N$ means that for a total budget $N$ of particles, one can consider any decomposition (as a power of $N$) of the particles between islands and within each island.



\section{Example of application} \label{sec:application}
In order to give illustration of this algorithm and of the previous theoretical results obtained, we present in this section two estimation problems.\\
First let us recall the example of a mobile whose evolution is influenced by an unknown force which has been described in \eqref{mobile}. Noisy observations of this physical systems are available. We resume the dynamics by the following system of equations :
\begin{align}
 \left\lbrace\begin{array}{lll}     
              X_{n+1} &=& X_{n} + V_n \left(\begin{array}{c}\cos\a  \\\sin\a \end{array}\right)\Delta t + \Theta_{n+1} \Delta t + B^X_n\\
	      V_{n+1} & = & V_n + B^V_n \\
              Y_{n} &=& h(X_n,V_n) + B^Y_n
             \end{array}
\right.
\end{align}
where $X_{n+1}$ denotes the position of the mobile in the plane, $V_n$ the proper speed of the mobile and $Y_{n}$ their noisy observations through the observation function $h$, with $B^Y_n\sim \normdist(0,\Sigma^{Y})$. The course track of the mobile $\a$ is constant over time. The vector $\Theta_n$ is a random variable and denotes the unknown force acting on the position of the mobile. Its equation of evolution is given by 
$$\Theta_{n+1}=\left(\begin{array}{c}\Theta^1_{n+1}\\\Theta^2_{n+1}\end{array}\right)=\left(\begin{array}{c}\cos\Theta^1_{n}\\\sin\Theta^2_n\end{array}\right)+B^{\Theta}_n$$
with $B^{\Theta}_n\sim\normdist(0,\Sigma^{\Theta})$.
The initial condition of the system is given by $X_0\sim \normdist(m^X_{\t_0,0},\Sigma^X_{\t_0,0})$, $V_0\sim \normdist(m^V_{0},\Sigma^V_{0})$ and $\a=\pi/2$.
We are interested in the estimation of the position of the mobile, which depends on the parameter $\Theta_n$.  We thus need to learn both the force, the speed and the position of the mobile. The tricky part is that there is no observation of the force. 
Here we will consider that the speed is a Poisson process, that is $B^V_n$ is a Poisson process of intensity 0.03 where the jumps high is given by a standard normal distribution of variance 3. Concerning $B^X_n$, it is a Gaussian random variable such that $B^X_n\sim\normdist(0,\Sigma^X)$.
We present now the results obtained for a simulating time of 125 minutes with $\Delta t=15s$. The value of the different variances are set to $$\Sigma^{\theta}=\left(\begin{array}{c c}1&0\\0&1\end{array}\right),\ \Sigma^{X}=\Sigma^X_{\t_0,0}=\left(\begin{array}{c c}1.5&0\\0&1.5\end{array}\right),\text{ and }\Sigma^{Y}=\left(\begin{array}{c c c}0.5&0&0\\0&0.5&0\\0&0&1\end{array}\right).$$ As one can notice, to estimate the law of the couple $(\Theta_n,\kacX{n})$ given the observations $Y_{0:n}$,
one can use Interacting Kalman filters and labeled island particle filters (LIPFs), detailed respectively in Algorithms \ref{IKF} and \ref{algorithme island}. We present comparative results obtained thanks to both methods. \\
Concerning the labeled version, the potential of each particle is given by the density of the observations, that is for all $x_n\in\spx{n}$ and for all $\t_n\in\spt{n}$:
$$ G_{n}(\theta_n,x_n)\propto \exp\left(-\frac{1}{2}(y_n-h(x_n,v_n))^T(\Sigma^{Y})^{-1}(y_n-h(x_n,v_n))\right).$$
On all the figures the realization of the true signal is represented by the color black, the observations $Y$ are represented by the color blue, the filtered signal obtained thanks to Algorithm \ref{algorithme island} with $\Nis=100$ and $\Nin=300$ in red and results obtained using Algorithm \ref{IKF} in green with $\Nis=100$.
On \autoref{estimation_vitesse}, one realization of the signal $V_n$, its observed and its estimations counterparts are represented with respect to time. As one may observe, the true signal is well estimated by the technique we develop. Indeed, here the Interacting Kalman filter is not optimal as the noise sequence is not Gaussian.
On \autoref{estimation force strength}, we represent the temporal evolution of the force strength estimation. One can notice that even if no observation is  available, we are able to find back the value of the true signal thanks to Algorithm \ref{algorithme island} whereas Algorithm \ref{IKF} retrieves only a global trend. 
\begin{figure}[H]
\begin{center}
\includegraphics[width=\textwidth]{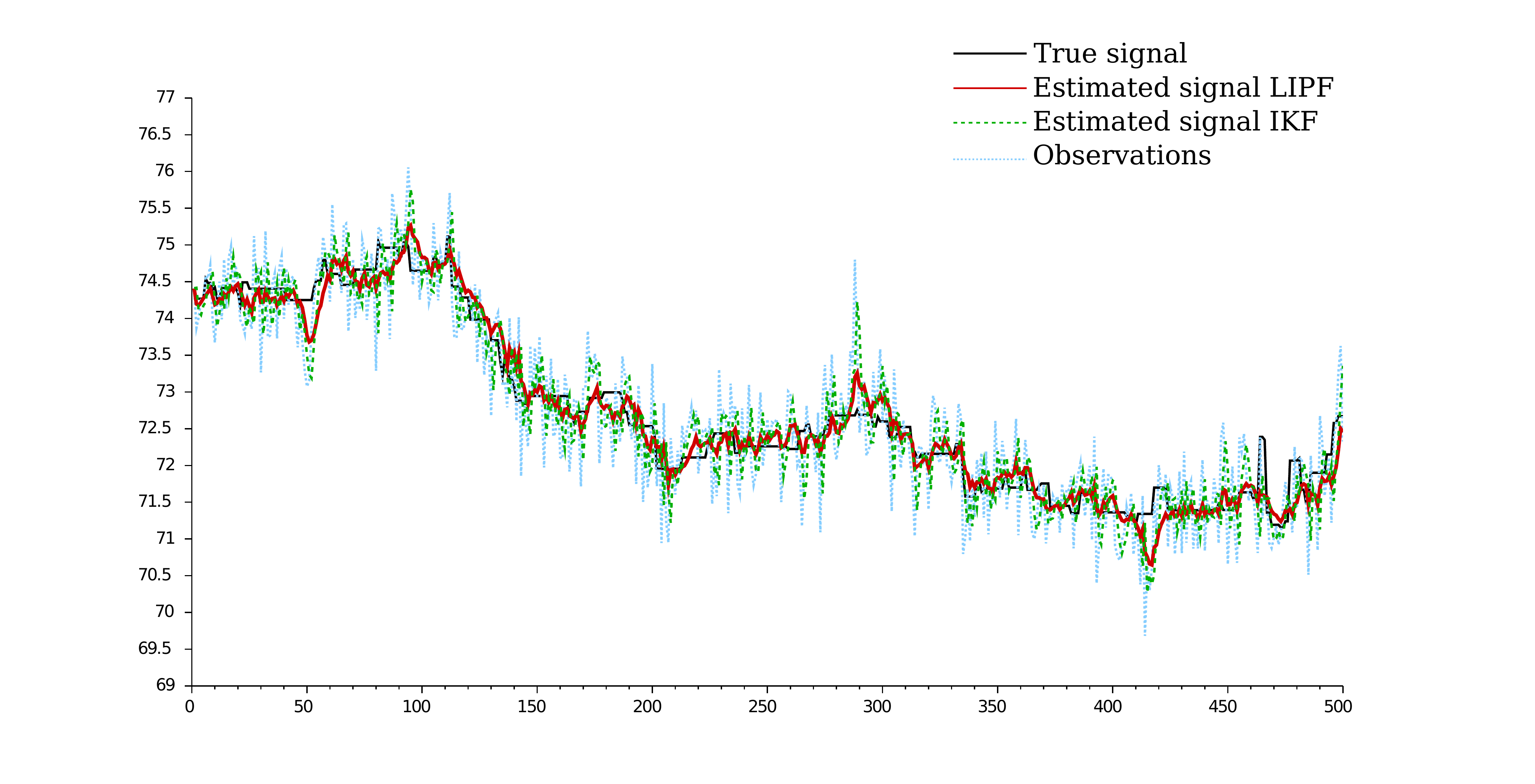}
\caption{Temporal evolution of the mobile's speed, its observed and filtered counterparts.}
\label{estimation_vitesse}
\end{center}
\end{figure}
\autoref{direction} represents the temporal evolution of one realization of the force orientation and its estimated counterparts. Results obtained thanks to Algorithm \ref{algorithme island} give a better estimation of the true signal than the results obtained thanks to the Algorithm \ref{IKF}.
\begin{figure}[H]
\begin{minipage}[b]{0.50\linewidth}
  \includegraphics[height=4cm,width=7cm]{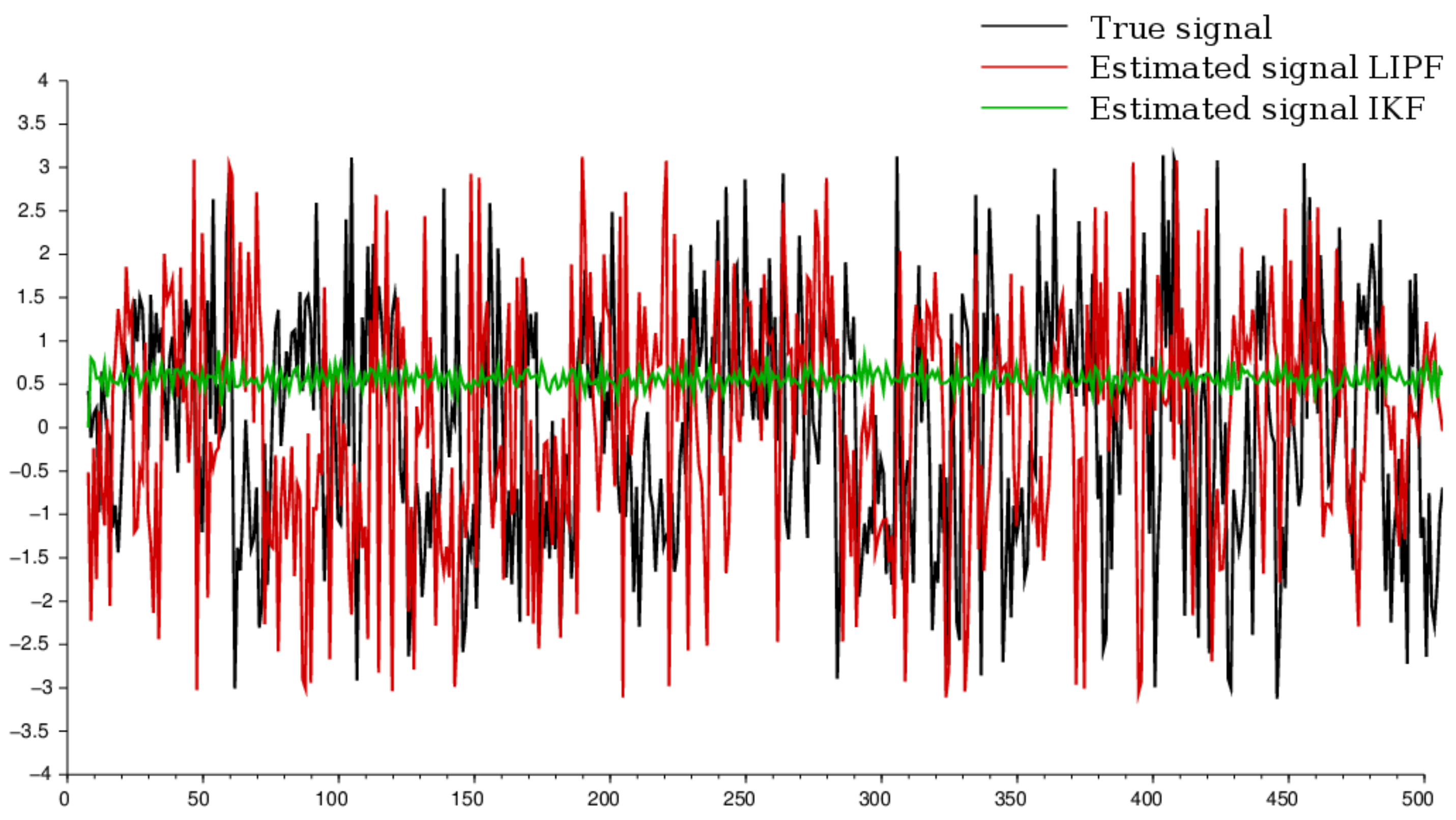}
  \caption{Temporal evolution of the force orientation in $rad$}
  \label{direction} 
\end{minipage}
 \hspace{0.5cm}
\begin{minipage}[b]{0.45\linewidth} 
\includegraphics[height=4cm,width=6cm]{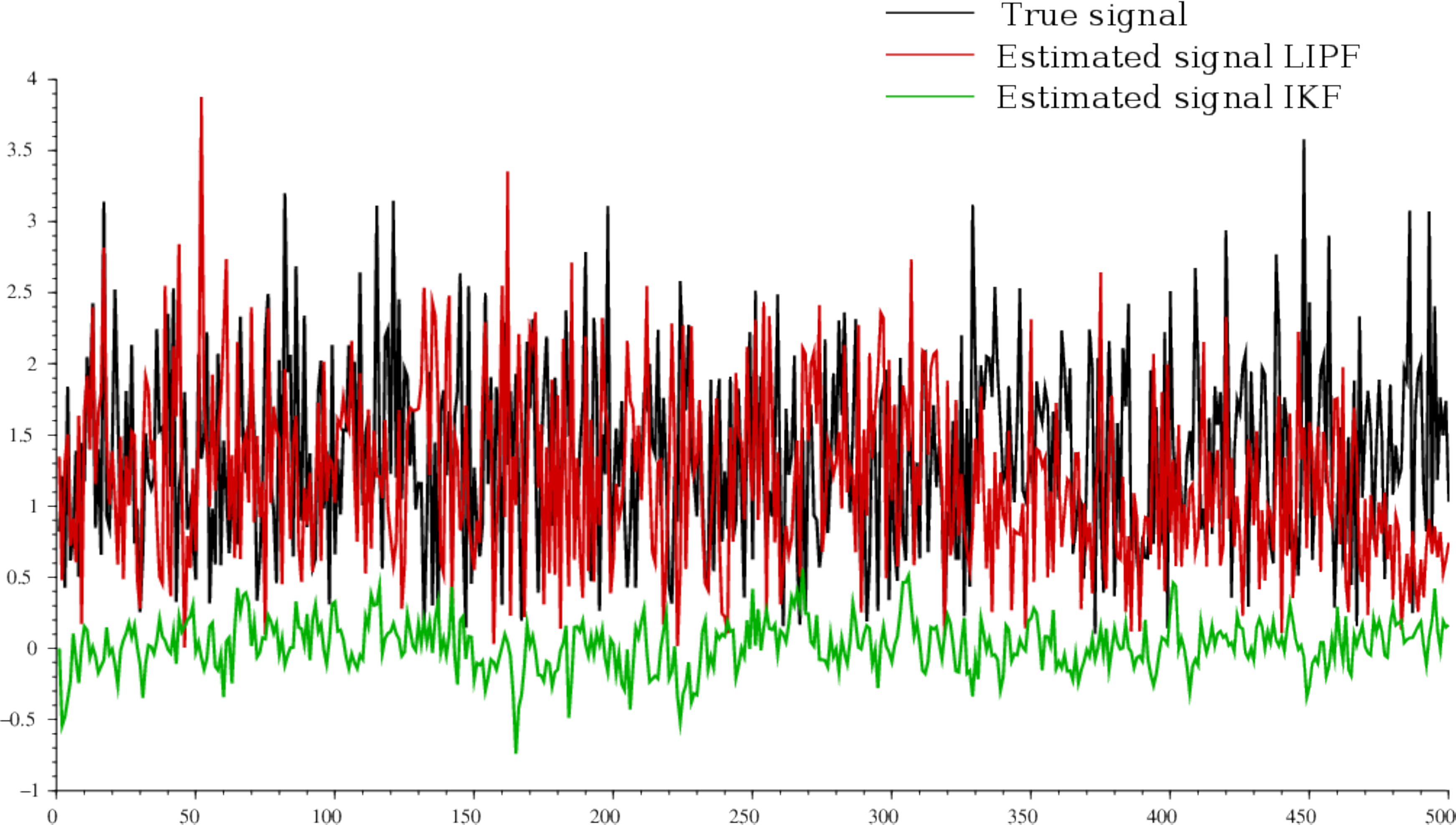}
\caption{Temporal evolution of the force strength in $m.s^{-1}$}
\label{estimation force strength}
\end{minipage}
\end{figure}
From this example we can conclude that the labeled island particle filter is able to filter observations of the process while estimating the environment where the process evolves. Moreover the comparison with the Interacting Kalman filter algorithm shows that the labeled island particle filter is more effective to treat this double level estimation problem.
~~\\
Let us consider the 2-D filtering problem inspired from the growth model \cite{Kitagawa:1987}. This model, which is a standard benchmark example in the particle filtering literature, is given by the following system of equations :
\begin{align*}
 \left\lbrace\begin{array}{lll}
              \Theta_{n+1} & = & 8\cos(1.2(n+1))+B_{n+1}^{\theta}\\
              X_{n+1} &=& \displaystyle\frac{X_n}{2}+25\displaystyle\frac{X_n}{1+X_n^2}+\Theta_{n+1}+B_{n+1}^X \\
              Y_{n} &=& X_n+B^Y_n
             \end{array}
\right.
\end{align*}
where $\Theta_0\sim\normdist(0,\sigma_{\t}^2)$, $X_0\sim\normdist(0,\sigma_{X}^2)$, $B_{n+1}^{\t}\sim\normdist(0,\sigma_{\t}^2)$, $B_{n+1}^{X}\sim\normdist(0,\sigma_X^2)$ and $B^Y_{n}\sim\normdist(0,\sigma_Y^2)$.\\
We use the labeled island particle model to estimate the law of the couple $(\Theta_n,\kacX{n})$ given the observations $Y_{0:n}$, where the potential functions $G_n$ are given by the likelihood of the observations, that is for all $x_n\in\spx{n}$ and $\t_n\in\spt{n}$:
$$ G_{n}(\theta_n,x_n)\propto \exp\left(-\frac{(Y_n-x_n)^2}{2\sigma_Y^2}\right) \eqsp. $$
We present the results obtained for a simulating time of $1000$ time steps. 
The different variances are set to $\sigma_{\t}^2=1$, $\sigma_{X}^2=1$ and $\sigma_{Y}^2=10$. 
On all the figures the realization of the true signal is represented in black color, the observations $Y$ are represented in blue, and the filtered signal obtained thanks to Algorithm \ref{algorithme island} with $\Nis=200$ and $\Nin=100$ is represented in red. 
On Figure \ref{theta example}, one realization of the signal $\Theta$ and its estimation obtained thanks to the labeled island particle algorithm are represented on a small period of time. As one may observe, the true signal is well estimated even if no observations are available.
On Figure \ref{x example}, one realization of the process $X$ is represented, its observed and its estimation counterparts. Even if the observations are really noisy, one is able to filter out the noise to find back the value of the true signal.
 \begin{figure}[H]
 \begin{minipage}[b]{0.45\linewidth}
\centering
\includegraphics[width=\textwidth]{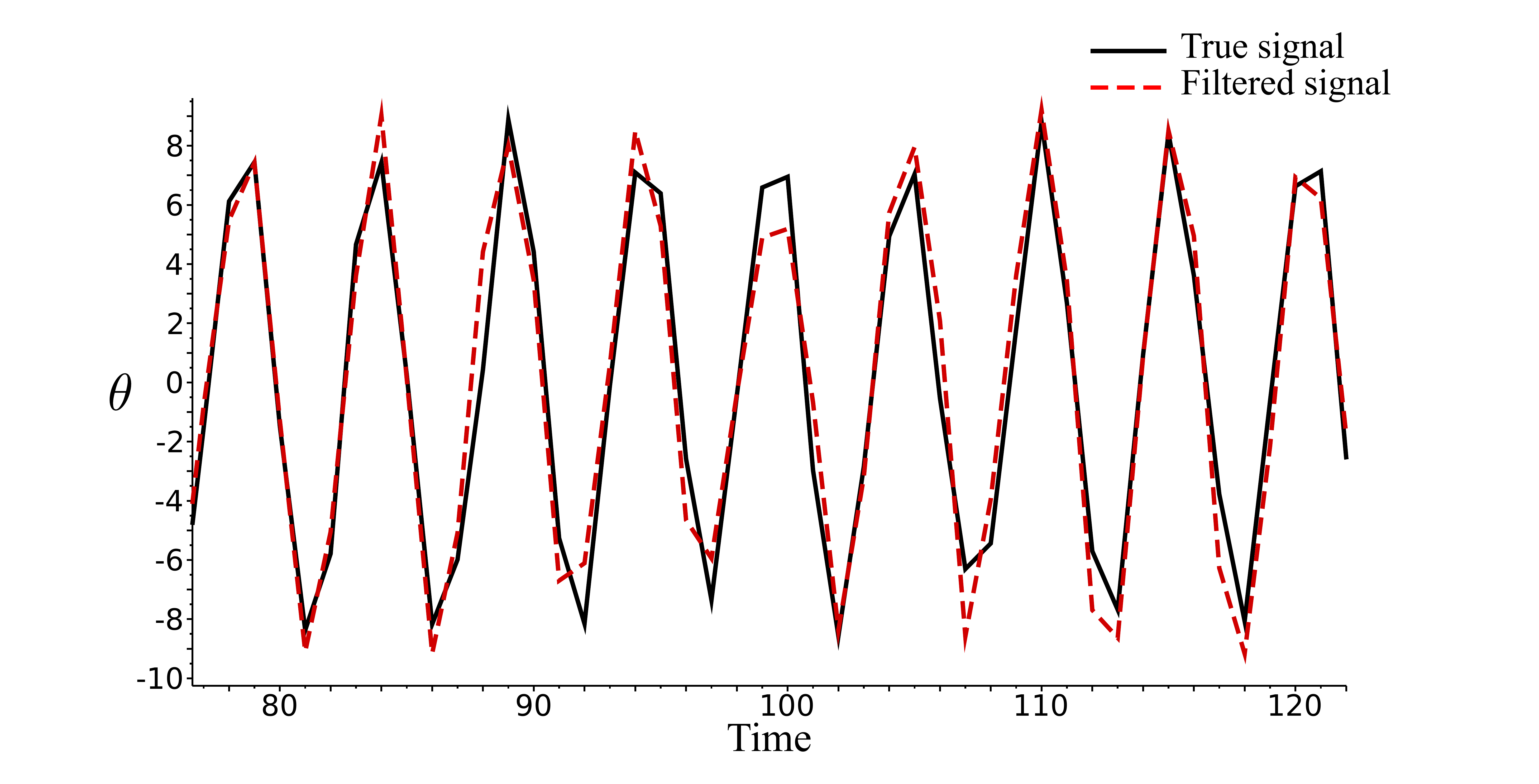}
\caption{Temporal zoom on one realization of the $\Theta$ process and its estimated counterpart}
\label{theta example}
\end{minipage}
\hspace{0.5cm}
\begin{minipage}[b]{0.45\linewidth}
\centering
\includegraphics[width=\textwidth]{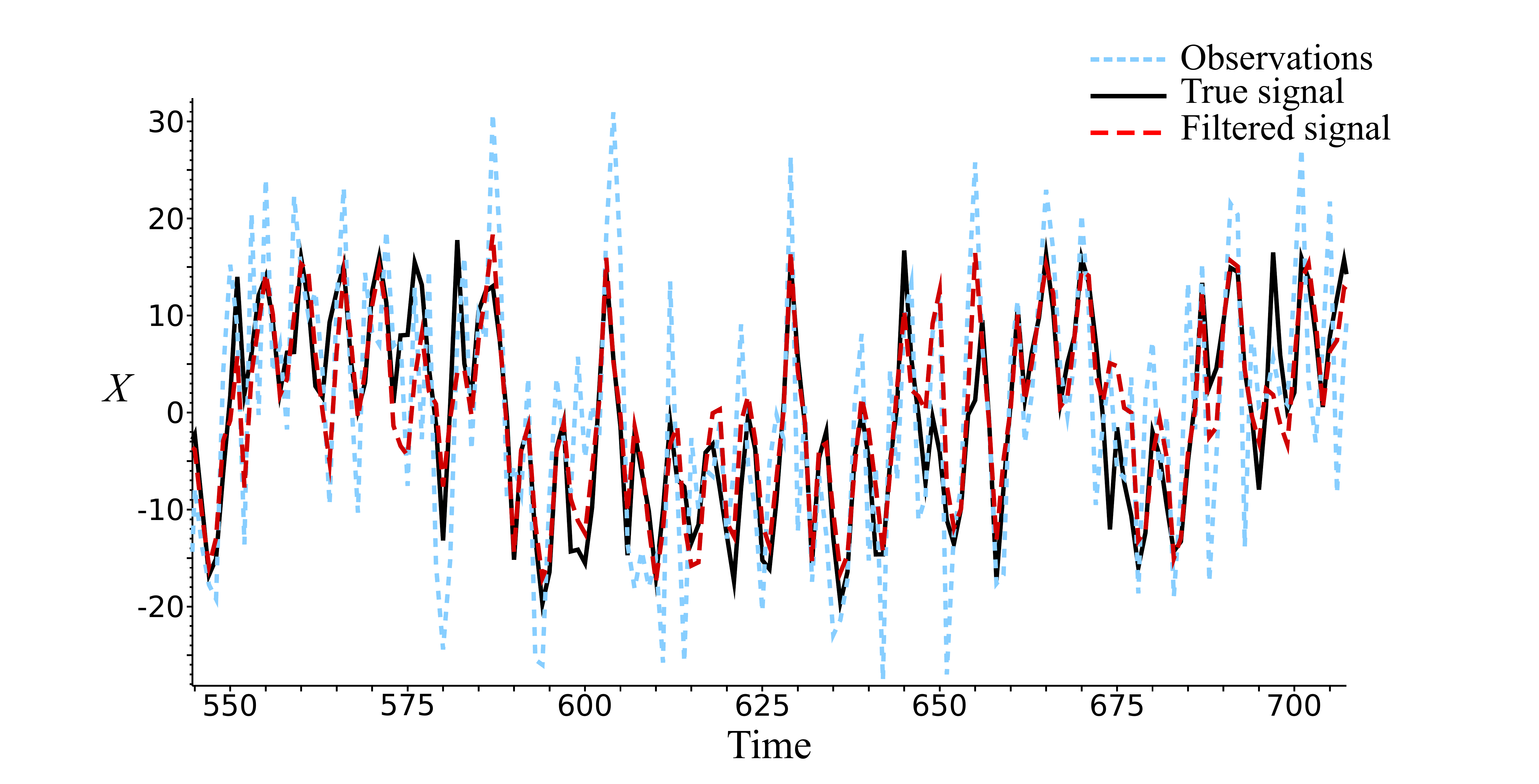}
\caption{Temporal zoom on one realization of the $X$ process, its observed and estimated counterparts}
\label{x example}
\end{minipage}
 \end{figure}
\begin{figure}[H]
 \begin{center}
  \includegraphics[scale=0.25]{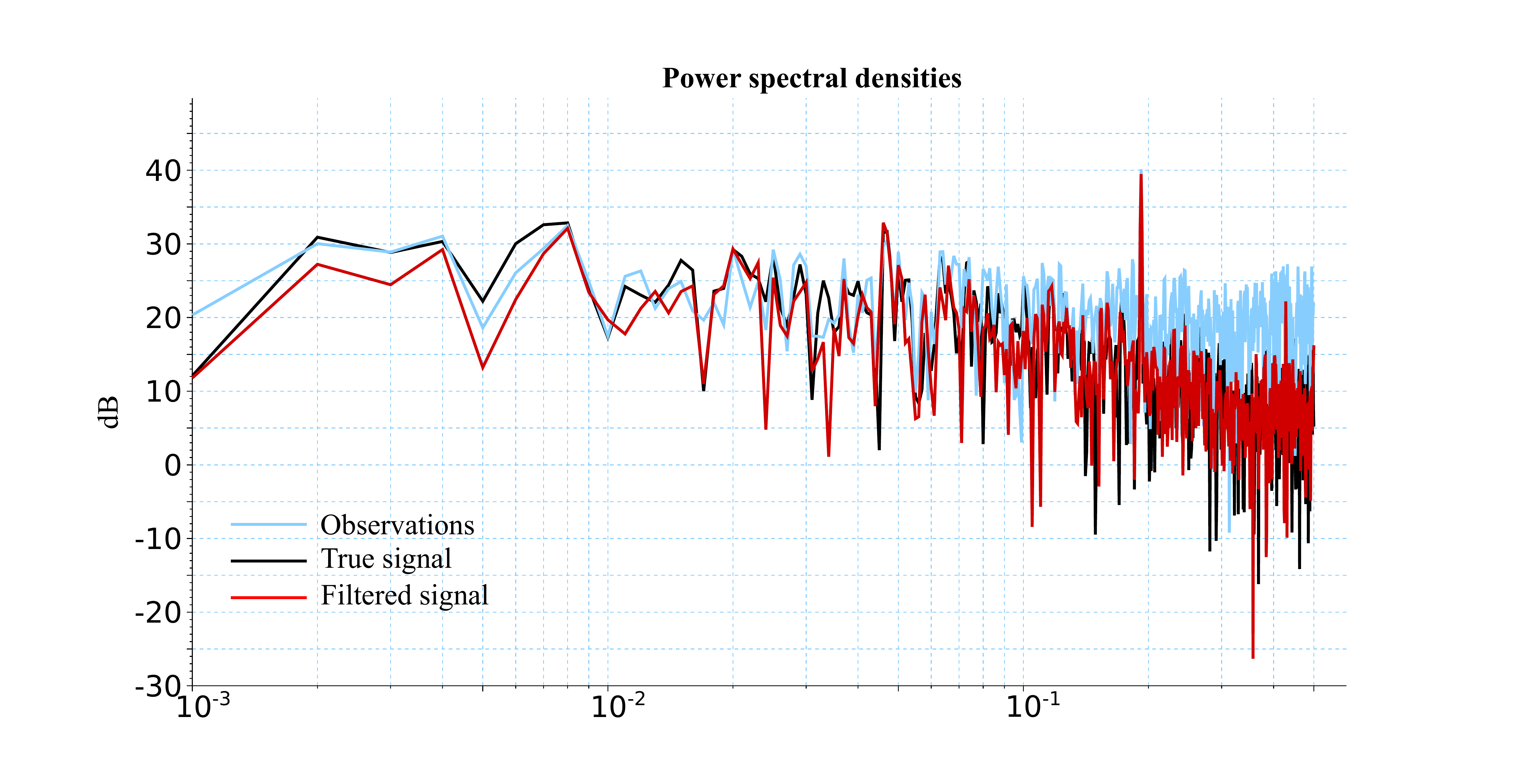}
  \caption{True, filtered and observed power spectral densities for one realization of the process $X$ over 1000 time steps.}
  \label{dsp example}
 \end{center}
\end{figure}
Indeed, as one may have noticed, on Figure \ref{dsp example}, the filtered power spectral density (in red) is closer to the black line, representing the ``true'' signal, than the observed power spectral density which has the same shape as a white noise for the high frequencies. Moreover, some frequencies are found even if there are not present in the observed signal. These two observations illustrate the convergence of the estimator constructed by the \lip ~algorithm detailed in Algorithm \ref{algorithme island}.

Then we run 100 times the same experiment to get a sample of realizations for the true signal and the filtered signal. In that way one can illustrate the theoretical results obtained for the $\Lp$ error bound.
On figures \ref{erreur theta example} and \ref{erreur x example} are presented the $\mathbb{L}^2$ errors between the estimated law and the true law at one time step respectively for $\Theta$ and $X$ in function of the number of islands $\Nis$ and the number of particles inside each island $\Nin$. This error decreases both with the number of particles and the number of islands as it was suggested by the \autoref{theo:Lp bound}.
\begin{figure}[H]
 \begin{minipage}[b]{0.45\linewidth}
\centering
\includegraphics[width=\textwidth]{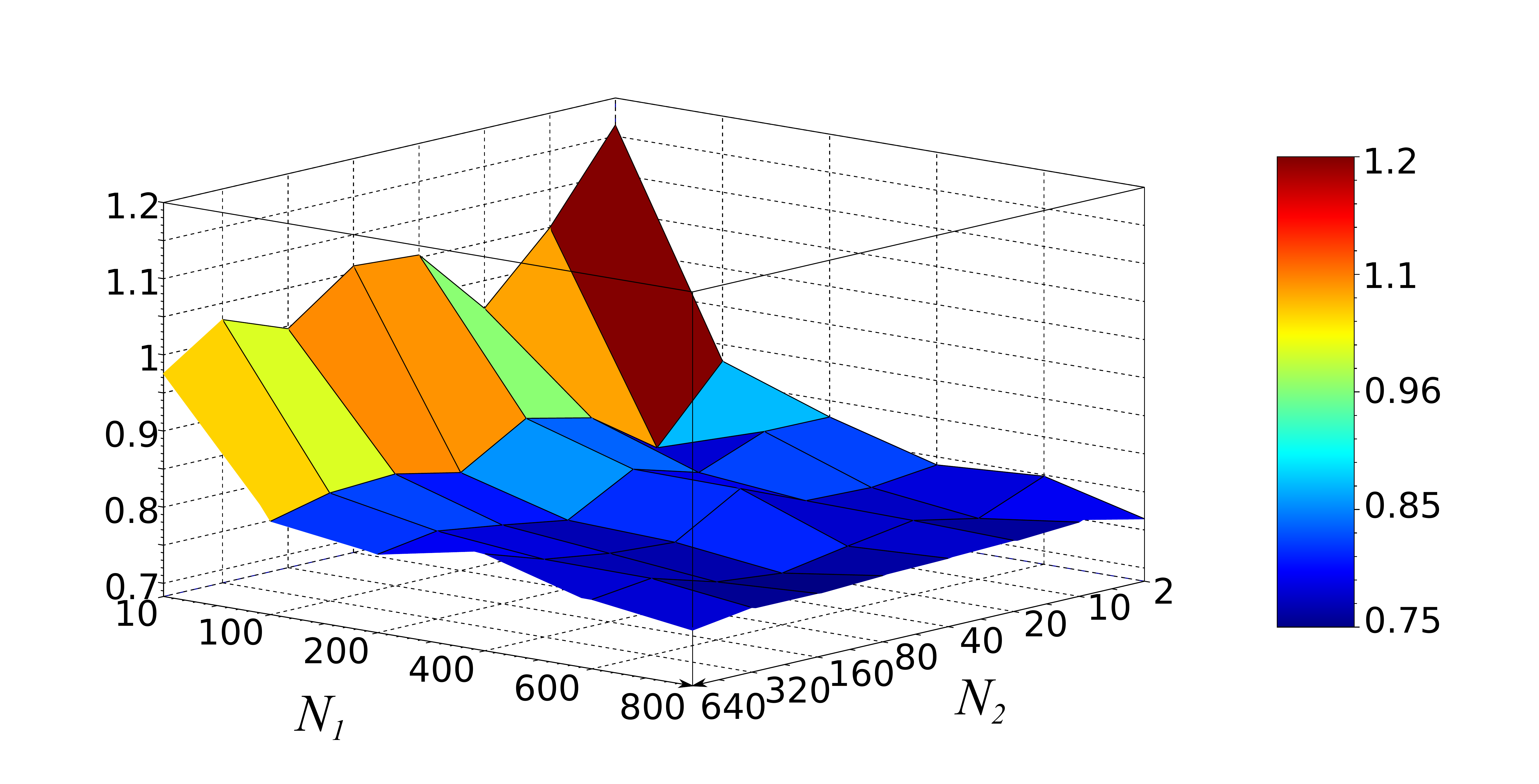}
\caption{Evolution of the estimation's error for the law of $\Theta$ for 100 realizations of the process in function of $\Nis$ and $\Nin$}
\label{erreur theta example}
\end{minipage}
\hspace{0.5cm}
\begin{minipage}[b]{0.45\linewidth}
\centering
\includegraphics[width=\textwidth]{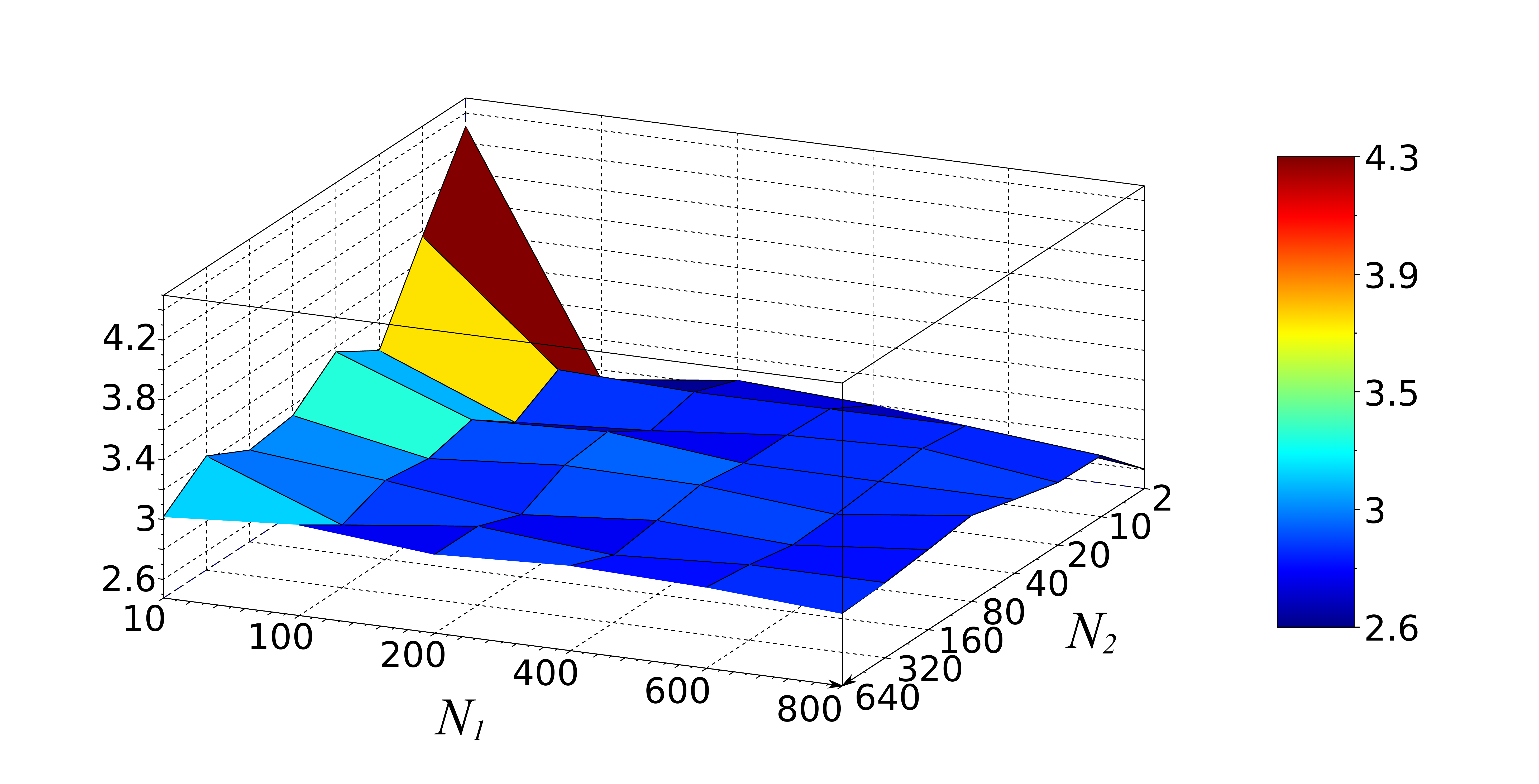}
\caption{Evolution of the estimation's error for the law of $X$ for 100 realizations of the process in function of $\Nis$ and $\Nin$}
\label{erreur x example}
\end{minipage}
 \end{figure}
Concerning the variance of the error made between the true law and the filtered one, on figures \ref{var theta example} and \ref{var x example} for $\Theta$ and $X$ respectively, one can observe that the results obtained in \autoref{thm:asconv} are confirmed. 
Moreover one can notice that the variance is more influenced by the number of islands than the number of particles inside each island. Indeed as in \autoref{var x example}, the variance obtained for a fixed time step is varying with respect to the number of islands and number of particles inside each islands. But if the number of islands influences the variance, we can observe that the number of particles inside each island does not seem to be really influent for a given number of islands.
 \begin{figure}[H]
 \begin{minipage}[b]{0.45\linewidth}
\centering
\includegraphics[width=\textwidth]{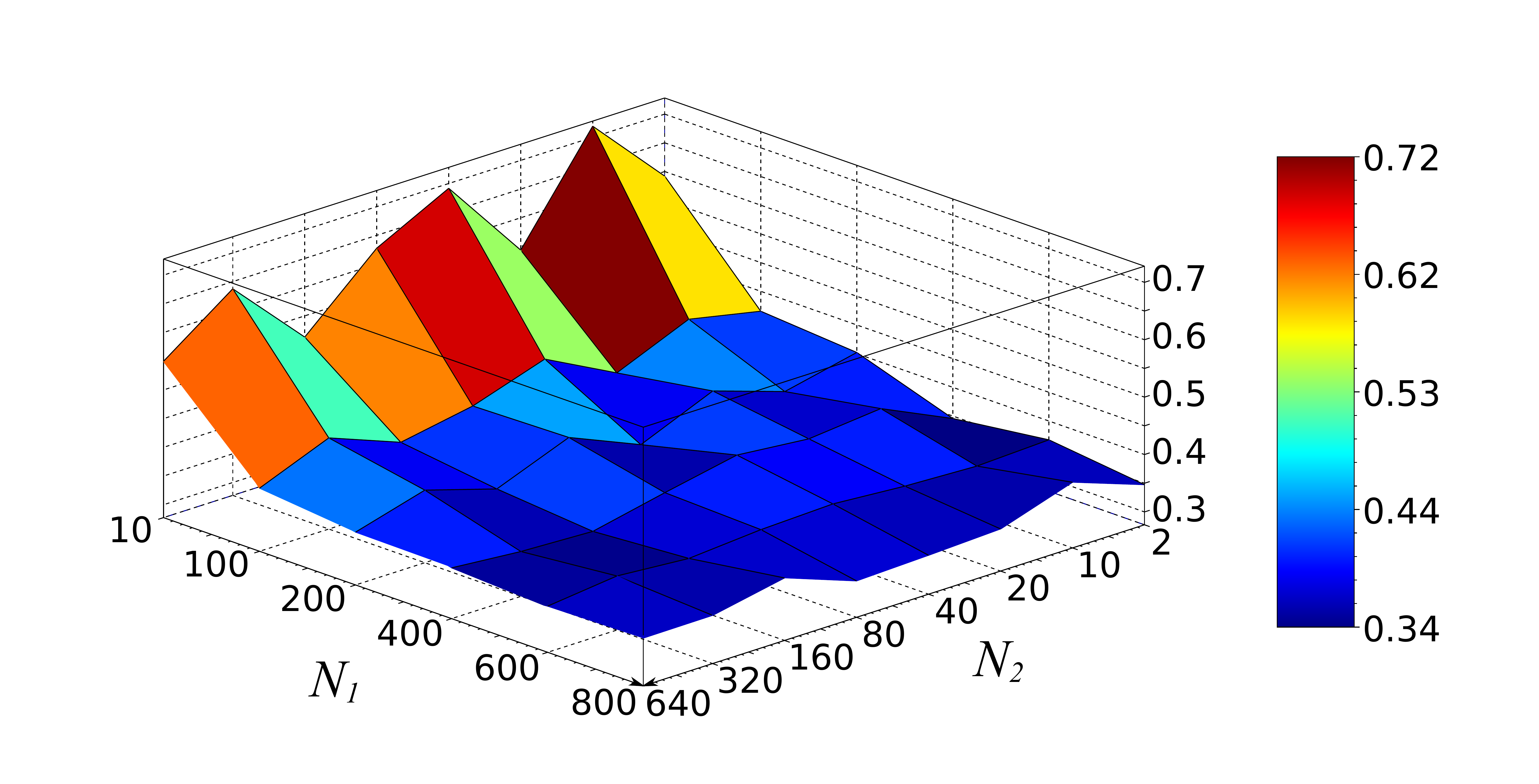}
\caption{Evolution of the variance estimation's error for the law of $\Theta$ for 100 realizations of the process in function of $\Nis$ and $\Nin$}
\label{var theta example}
\end{minipage}
\hspace{0.5cm}
\begin{minipage}[b]{0.45\linewidth}
\centering
\includegraphics[width=\textwidth]{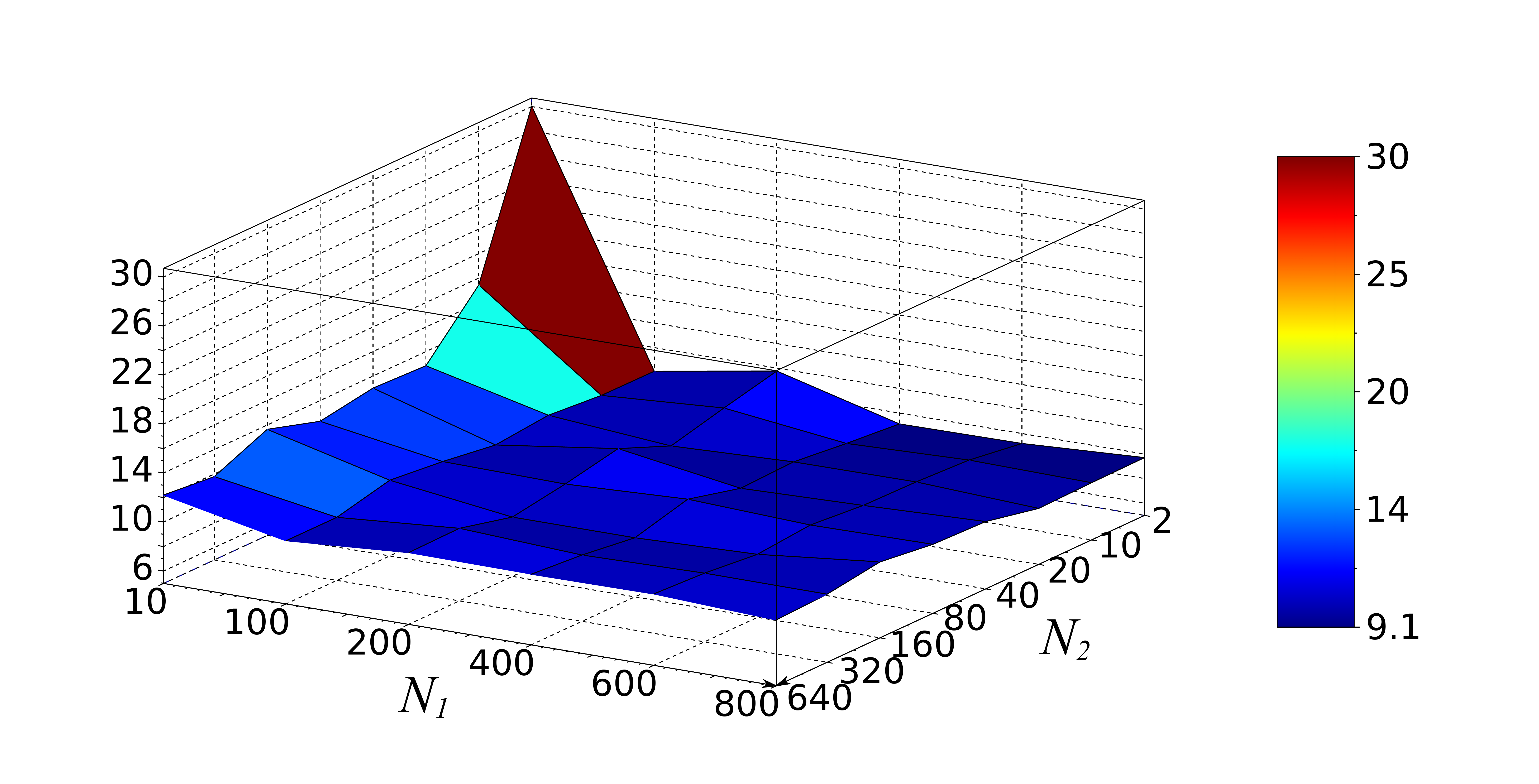}
\caption{Evolution of the variance estimation's error for the law of $X$ for 100 realizations of the process in function of $\Nis$ and $\Nin$}
\label{var x example}
\end{minipage}
 \end{figure}

\bibliographystyle{IEEEtran}
\bibliography{biblio}

\end{document}